\documentclass[a4paper, 10pt]{amsart}

\usepackage[square, numbers, comma]{natbib} 

\usepackage[usenames,dvipsnames,svgnames,table]{xcolor}
\usepackage{amsmath,amsthm,amssymb,amssymb,esint,verbatim,tabularx,graphicx}
\usepackage{fancyhdr}
\usepackage{enumerate}
\usepackage[normalem]{ulem}

\usepackage{amsmath}
\usepackage{fancyhdr}
\usepackage{epic}
\usepackage{pgf,tikz}
\usetikzlibrary{arrows}

\usepackage[utf8]{inputenc}
\usepackage{color}
\usepackage{hyperref}
\usepackage{verbatim}
\usepackage{pdfpages}

\hypersetup{urlcolor=blue, colorlinks=true} 

\newtheorem{theorem}{Theorem}[section]

\newtheorem{lemma}[theorem]{Lemma}
\newtheorem{proposition}[theorem]{Proposition}
\newtheorem{corollary}[theorem]{Corollary}
\newtheorem{definition}{Definition}[section]

\theoremstyle{remark}
\newtheorem{remark}[theorem]{Remark}
\theoremstyle{definition}
\newtheorem{example}{Example}[section]

\numberwithin{equation}{section}

\newcommand{\R}{\ensuremath{\mathbb{R}}}
\newcommand{\N}{\ensuremath{\mathbb{N}}}

\newcommand{\Z}{\ensuremath{\mathbb{Z}}}
\newcommand{\J}{\ensuremath{\mathbb{J}}}

\newcommand{\Levy}{\ensuremath{\mathcal{L}}}

\newcommand{\Operator}{\ensuremath{\mathfrak{L}^{\sigma,\mu}}}

\newcommand{\Hvar}{\ensuremath{\vartheta}}  

\newcommand{\Levymu}{\ensuremath{\mathcal{L}^\mu}}

\newcommand{\veps}{\varepsilon}

\newcommand{\Ti}{\ensuremath{T^{\textup{imp}}}}

\newcommand{\dd}{\,\mathrm{d}}

\newcommand{\dell}{\partial}
\newcommand{\indikator}{\mathbf{1}_{|z|\leq 1}}
\newcommand{\indik}{\mathbf{1}}

\DeclareMathOperator{\sgn}{\textup{sign}}

\newcommand{\Ux}{\overline{U}}

\newcommand{\Uxt}{\widetilde{\overline{U}}}
\newcommand{\Grid}{\mathcal{G}_h}
\newcommand{\GridT}{\mathcal{T}_{\Delta t}^T}

\newcommand{\RN}{\mathbb{R}^N}

\begin{document}

\title[Uniform tail estimates]{Uniform tail estimates and $L^p(\R^N)$-convergence for finite-difference approximations of nonlinear diffusion equations}

\author[F.~del~Teso]{F\'elix del Teso}
\address[F. del Teso]{Departamento de Matem\'aticas\\
Universidad Aut\'onoma de Madrid (UAM)\\
Campus de Cantoblanco, 28049 Madrid, Spain}
\email[]{felix.delteso\@@{}uam.es}
\urladdr{https://sites.google.com/view/felixdelteso/}

\author[J.~Endal]{J\o rgen Endal}
\address[J. Endal]{Departamento de Matem\'aticas\\
Universidad Aut\'onoma de Madrid (UAM)\\
Campus de Cantoblanco, 28049 Madrid, Spain}
\email[]{jorgen.endal\@@{}uam.es}
\urladdr{http://verso.mat.uam.es/~jorgen.endal/}

\author[E.~R.~Jakobsen]{Espen R. Jakobsen}
\address[E. R. Jakobsen]{Department of Mathematical Sciences\\
Norwegian University of Science and Technology (NTNU)\\
N-7491 Trondheim, Norway} 
\email[]{espen.jakobsen\@@{}ntnu.no}
\urladdr{http://www.math.ntnu.no/\~{}erj/}





\begin{abstract}
We obtain new equitightness and $C([0,T];L^p(\R^N))$-convergence results for finite-difference approximations of generalized porous medium equations of the form
$$
\partial_tu-\mathfrak{L}[\varphi(u)]=g\qquad\text{in $\R^N\times(0,T)$},
$$
where $\varphi:\R\to\R$ is continuous and nondecreasing, and $\mathfrak{L}$ is a local or
nonlocal diffusion operator. Our results include
slow diffusions, strongly degenerate Stefan problems, and fast diffusions above a critical exponent.  These results improve the previous $C([0,T];L_\textup{loc}^p(\R^N))$-convergence obtained in a series of papers on the topic by the authors. To have
equitightness and global $L^p(\R^N)$-convergence, some  additional restrictions
on $\mathfrak{L}$ and $\varphi$ are needed. Most commonly used
symmetric operators $\mathfrak{L}$ are still included: the Laplacian,
fractional Laplacians, and other generators of symmetric L\'evy processes
with some fractional moment. We also discuss extensions to nonlinear
 possibly strongly degenerate
convection-diffusion equations.
\end{abstract}

\maketitle

\begin{center}
  \emph{Dedicated to Juan Luis V\'azquez\\---with deep admiration---\\ on the occasion of his 75th birthday.}
\end{center}

\tableofcontents

\section{Introduction}
In this paper we study finite-difference approximations of
degenerate nonlinear diffusion and convection-diffusion equations. We
work in a standard low regularity setting with bounded and integrable data and
(generalized) solutions in $C([0,T];L^p(\R^N))$. Even though the
natural mode of convergence here is in the
  space where solutions live, existing results  for such schemes give
convergence in $C([0,T];L_\textup{loc}^p(\R^N))$ or
$L_\textup{loc}^p(\R^N\times(0,T))$. 
  The purpose of this paper is to show in a systematic and general way
  that with some additional effort
  we will have the stronger convergence in
  $C([0,T];L^p(\R^N))$ under mild or sometimes even no additional assumptions.
  
To be more precise we study generalized porous medium equations (and
extensions to convection-diffusion equations) in the context of bounded and integrable solutions. In detail, we study
\begin{equation}\label{GPME}\tag{\textup{GPME}}
\begin{cases}
\dell_tu-\mathfrak{L}[\varphi(u)]=g \qquad\qquad&\text{in $Q_T:=\R^N\times(0,T)$},\\
u(\cdot,0)=u_0 \qquad\qquad&\text{on $\R^N$},
\end{cases}
\end{equation}
where $u=u(x,t)$  is the solution, $\varphi:\R\to\R$ is nondecreasing and $\gamma$-H\"older continuous with $\gamma\in(0,1]$, $g=g(x,t)$ some bounded and integrable right-hand side, $u_0\in L^1(\R^N)\cap L^\infty(\R^N)$, and $T>0$. The general operator $\mathfrak{L}$ is given~as 
\begin{equation}\label{defbothOp}
\mathfrak{L}=\Operator:=L^\sigma+\Levy^\mu
\end{equation} 
where $L^\sigma:C_{\textup{c}}^\infty(\R^N)\to L^1(\R^N)$ is a possibly degenerate local diffusion operator
\begin{equation}\label{defLocalOp}
L^\sigma[\psi](x):=\textup{tr}\big(\sigma\sigma^TD^2\psi(x)\big)
\end{equation}
where $\sigma=(\sigma_1,\ldots,\sigma_P)\in \R^{N\times P}$, $P\in\N$, and $\sigma_i\in\R^N$, and the anomalous or nonlocal diffusion operator
$\Levymu$ is defined for any $\psi \in  
C_\textup{c}^\infty(\R^N)$ as 
\begin{equation}\label{deflevy}
\Levy ^\mu [\psi](x)=\int_{\R^N\setminus\{0\} } \big(\psi(x+z)-\psi(x)-z\cdot D\psi(x)\indikator\big) \dd\mu(z),
\end{equation}
where $D$ is the gradient, $\indikator$ a characteristic function, and
$\mu$ a nonnegative and symmetric Radon measure satisfying at least
the usual L\'evy measure condition (see Section
\ref{sec:MainResults}). In this paper, to simplify, we will always choose either
$\sigma\equiv0$ or $\sigma\sigma^T=I_N$. That is, 
the local operator is given as $L^\sigma=c \Delta:=c\sum_{i=1}^N\dell_{x_ix_i}^2$ where
$c\in\{0,1\}$.

\begin{remark}\label{deflevyexplained}
\begin{enumerate}[{\rm (a)}]
\item Since $\mu$ will be symmetric, $\lim_{r\to0^+}\int_{r<|z|\leq 1}z\dd\mu(z)=0$,
and we have an equivalent definition of $\Levy^\mu$ in \eqref{deflevy}
in terms of a principal value integral:
$$
\Levy^\mu[\psi](x)=\lim_{r\to0^+}\int_{|z|>r}\big(\psi(x+z)-\psi(x)\big)\dd\mu(z)=\textup{P.V.}\int_{|z|>0}\big(\psi(x+z)-\psi(x)\big)\dd\mu(z).
$$
\item \eqref{GPME} is a pure diffusion equation. Our results also
  holds for convection-diffusion equations and viscous conservation
  laws, see Section \ref{sec:ResultsForCDE}.
\end{enumerate}
\end{remark}

Equations of the form \eqref{GPME} (and also variants with convection) appear in numerous applications. We selectively mention reservoir simulation, sedimentation processes, and traffic flow in the local case \cite{EsKa00, BuCoBuTo99, Whi74}; cardiac electrophysiology and semiconductor growth in the nonlocal case \cite{B-OKaGrRoBu14, Woy01}; and flows of fluids through porous media, mathematical finance, information theory, and phase transitions in both cases \cite{Vaz07, DPQuRoVa12, CoTa04, ShWe49, BrChQu12}. Important examples are strongly degenerate Stefan problems (cf. \cite{DTEnVa20,DTEnVa21}) with $\varphi(u)=\max\{0,u-L\}$, $L>0$, and the full range of porous medium and fast diffusion equations (cf. \cite{DPQuRoVa12}) with $\varphi(u)=u|u|^{m-1}$, $m>0$. The class of operators  $\Operator$ coincides with the generators of symmetric L\'evy processes  \cite{Ber96,Sch03,App09}. This includes e.g. the Laplacian $\Delta$, fractional Laplacians $-(-\Delta)^{\frac{\alpha}{2}}$, $\alpha\in(0,2)$, relativistic Schr\"odinger type operators $m^\alpha I-(m^2I-\Delta)^{\frac{\alpha}{2}}$, $\alpha\in(0,2)$ and $m>0$,  tempered stable processes \cite{CoTa04}, and even discretizations of all of these. Since $\sigma$ and $\mu$ may be degenerate or even identically zero, equation \eqref{GPME} can be purely nonlocal, purely local, or a combination. 

Our convergence results follow from compactness, stability, and
(very weak) uniqueness arguments. The compactness results depend on the type of equation under study, and
the properties available for such an equation. They are essential in
the context of existence, continuous dependence, and asymptotic
behaviour. For the latter example, this is particularly the case when
considering the rescaled solution in the ``four-step method''
introduced by Kamin and V\'azquez in \cite{KaVa88}. In all of these
cases, an approximate solution $u_n$ of the equation under study is
considered. Then one shows compactness of the family formed by $u_n$
in order to be able to find a limit function. The limit must of course
be a solution of the original equation (or variants of it). In this
paper, the approximate solution will always come from a
finite-difference scheme for \eqref{GPME}.

To prove compactness in $C([0,T];L^p(\R^N))$ with $p\in[1,\infty)$, we
  employ the well-known Arzel\`a-Ascoli and Kolmogorov-Riesz
  compactness theorems (cf. Appendix \ref{sec:AAKR}). A systematic
  approach to these theorems are presented in Section
  \ref{sec:abstractCompConv}. Compared to some previous results in
  this direction (see e.g. Lemma 2.2 in \cite{KaRi01} and Theorem A.8
  in \cite{HoRi02}), we also deduce equitightness (uniform
  tail control) for approximate solutions of \eqref{GPME}. This allows us to (i) go from $L^p_{\textup{loc}}$ 
  to $L^p$ compactness and (ii) to get a minimal and efficient
  compactness argument in the sense that we avoid a lot of previously used tricks to fulfil the requirements of Arzel\`a-Ascoli and Kolmogorov-Riesz. In our generality uniform boundedness ($L^\infty$-stability) of
  the solutions are   still required  to  ensure 
  some estimates hold. In special cases however, this assumption probably
  can be dropped.

As far as we know, equitightness results for \eqref{GPME} have not been presented under our general assumptions before (especially in the nonlocal case), and these results are really the core of this paper. Such estimates are deduced by taking, roughly speaking, $x\mapsto\mathbf{1}_{|x|>R}$ as a test function in the very weak formulation of \eqref{GPME}. This gives uniform tail control of the approximative solution (cf. e.g. \cite[Proof of Proposition 10.2]{DPQuRoVa12}). In examples when \eqref{GPME} conserves mass, we can summarize our equitightness results by saying that such estimates always holds in the local case, and also in the nonlocal setting when we assume that the nonlocal operator is comparable to a  fractional order operator at infinity.

Finite-difference methods were developed in the full generality of \eqref{GPME} in \cite{DTEnJa18b, DTEnJa19}, where also relevant numerical simulations can be found. Some early works are \cite{Dro10, DrJa14}. We also refer to \cite{DrEy16} (see also Part II of \cite{DrEyGaGuHe18}) in the purely local case. In the case of \eqref{GPME} with an additional convection term, we mention e.g. \cite{KaRi01, HoRi02, CiJa14}.

The numerical schemes which will be presented below include most of
the mentioned works on local and nonlocal cases. However, none of the
above showed convergence in $C([0,T];L^p(\R^N))$ (but some could still
deduce that the limit itself belonged to that space, see also Section
\ref{sec:ImprovedLimitRegularityLip}). Our equitightness estimates
therefore improves convergence results already present in the
literature. Conservative and monotone finite-difference schemes for
scalar conservation laws are discussed e.g. in \cite[Theorem
  3.8]{HoRi02}, and they immediately fall into our
$C([0,T];L^p(\R^N))$-framework. By 
adding a possibly nonlinear local diffusion term to such equations, we
obtain convection-diffusion equations with local diffusion. Such
equations have been studied in the context of finite-difference
approximations in \cite[Theorem 4.2]{KaRi01}. See also Theorem 3.9 and
Corollary 3.10 in \cite{EvKa00}. Again, we can improve the convergence
from the respective $L_\textup{loc}^1(Q_T)$ and
$C([0,T];L_\textup{loc}^1(\R^N))$ to $C([0,T];L^1(\R^N))$ since
$L^1(\R^N)$-equitightness holds. In the nonlocal
diffusion setting, finite-difference schemes has just recently been
analyzed in rigorous detail and generality in \cite{DTEnJa18b,
  DTEnJa19}, see also \cite{Dro10, DrJa14}. The former two references
obtain compactness/convergence results in
$C([0,T];L_\textup{loc}^1(\R^N))$, and the latter in
$L_\textup{loc}^1(Q_T)$. Our framework thus improves the
compactness/convergence results of all four papers.

Finally, we mention that our techniques do not give convergence rates of the scheme since we obtain
convergence via compactness plus uniqueness arguments. It is an
interesting and hard problem to show convergence rates in the general class
of equations studied here.
No general results are known up to now, both for \eqref{GPME} and the
convection-diffusion extension \eqref{CDE} discussed in Section
\ref{sec:ResultsForCDE}. The only known results we are aware of are
some hyberbolic type of convergence estimates obtained for {\em
  entropy solutions} via doubling of variables and regularisation
arguments, see \cite{CiJa14} for nonlocal and \cite{KaSt14,KaRiSt16}
for local equations. 

\subsubsection*{Outline} We start by reviewing known compactness theorems in Section \ref{sec:abstractCompConv}. Assumptions and extensions are discussed. Main results are provided in Section \ref{sec:MainResults} As the nonlocal operator $\Levy^\mu$ will be the hardest term to control uniformly at infinity, we discuss its discretization, needed assumptions, and related estimates in that section as well. Section \ref{sec:ProofOfEquitightness} is reserved for proofs, and Section \ref{sec:ExtensionsAndComments} for extensions (including the case of the convection) and comments. Important well-known results are presented in Appendices \ref{sec:AAKR} and \ref{sec:Minkowski}, and finally, some auxiliary results regarding equitightness can be found in Appendix \ref{sec:AuxRes}.

\subsubsection*{Notation}
Derivatives are denoted by $'$,
$\frac{d}{dt}$, $\dell_{x_i}$, and $D\psi$ and $D^2\psi$ denote the $x$-gradient and Hessian matrix of $\psi$. $\rho_\delta$ with $\delta>0$ will denote a standard mollifier.

We use standard notation for $L^p(\R^N)$, $W^{p,q}(\R^N)$, and $C_\textup{b}(\R^N)$. Moreover, $C_\textup{c}^\infty(\R^N)$ is the space of smooth functions with compact support, and $C([0,T];L_\textup{loc}^p(\R^N))$ the space of measurable functions $\psi:[0,T]\to L_\textup{loc}^p(\R^N)$ such that $\psi(t)\in L_\textup{loc}^p(\R^N)$ for every $t\in[0,T]$, $\sup_{t\in[0,T]}\|\psi(t)\|_{L^p(K)}<\infty$, and $\|\psi(t)-\psi(s)\|_{L^p(K)}\to0$ when $t\to s$ for all compact $K\subset \R^N$ and $t,s\in [0,T]$. In a similar way we also define $C([0,T];L^p(\R^N))$. Note that the notion of $\R^N\times(0,T)\ni(x,t)\mapsto \psi(x,t)\in C([0,T];L_\textup{loc}^p(\R^N))$ is a subtle one. In fact, we mean that $\psi$ has an a.e.-version which is continuous $[0,T]\to L_\textup{loc}^p(\R^N)$. See e.g. p. 726 in \cite{DrEy16} for more details. The space $C^{0,\Hvar}(\R)$ with $\Hvar\in(0,1]$ is identified as the Banach space with norm $\|\cdot\|_{C^{0,\Hvar}(\R)}:=\|\cdot\|_{C_\textup{b}(\R)}+|\cdot|_{C^{0,\Hvar}(\R)}$ where
$$
|\psi|_{C^{0,\Hvar}(\R)}:=\sup_{x,y\in\R}\left\{\frac{|\psi(x)-\psi(y)|}{|x-y|^{\Hvar}}\right\}.
$$
When $\Hvar=1$, we simply get $C^{0,\gamma}(\R)= W^{1,\infty}(\R)$.

\begin{remark}
From now on, we will study convergence in $C([0,T];L^r(\R^N))$ (abbrev. $C_t(L_x^r)$)  with $r\in[1,\infty)$. We use $r$ because we want it to be a unique identifier.
\end{remark}


\section{On compactness and convergence in $C([0,T];L^r(\R^N))$}
\label{sec:abstractCompConv}

Let us give an overview of the properties needed to deduce compactness and convergence in $C_t(L_x^r)$ with $r\in[1,\infty)$. 

\subsection{Necessary and sufficient conditions for compactness} Consider a sequence $\{u_n\}_{n\in\N}\subset C_t(L_x^r)$, and assume that $\{u_n\}_{n\in\N}$ satisfies: 
\begin{enumerate}[{\rm\bf (I)}]
\item \label{sec2prop:equitight}\textup{\bf (Equitight in space pointwise in time)} For all $t\in[0,T]$,
$$
\lim_{R\to\infty}\sup_n\int_{|x|> R}|u_n(x,t)|^r\dd x=0.
$$
\item \label{sec2prop:equicontspace}\textup{\bf (Equicontinuous in space pointwise in time)} For all $t\in[0,T]$ and all $\xi\in\R^N$, there exists a modulus of continuity $\lambda$ such that 
$$
\sup_n\|u_n(\cdot+\xi,t)-u_n(\cdot,t)\|_{L^r(\R^N)}\leq \lambda(|\xi|).
$$
\item \label{sec2prop:equiconttime}\textup{\bf (Equicontinuous in time)} For all $t,s\in[0,T]$, there exists a modulus of continuity $\lambda$ such that
$$
\sup_n\|u_n(\cdot,t)-u_n(\cdot,s)\|_{L^r(\R^N)}\leq \lambda(|t-s|).
$$
\end{enumerate}

Properties \eqref{sec2prop:equitight}--\eqref{sec2prop:equiconttime} are exactly the necessary and sufficient conditions needed to obtain compactness in $C_t(L_x^r)$ (cf. the Arzel\`a-Ascoli and Kolmogorov-Riesz compactness theorems \ref{thm:AA} and \ref{thm:KR}).

\begin{theorem}[Compactness]\label{metaCompactness}
Assume $r\in[1,\infty)$, and let 
\[\{u_n\}_{n\in\N}\subset C([0,T];L^r(\R^N)).\] 
The following are equivalent:
\begin{enumerate}[{\rm (a)}]
\item The sequence $\{u_n\}_{n\in\N}$ satisfies properties \eqref{sec2prop:equitight}--\eqref{sec2prop:equiconttime}.
\item There exists a subsequence $\{u_{n_k}\}_{k\in\N}$ and a $u\in C([0,T];L^r(\R^N))$ such that
$$
u_{n_k}\to u\qquad\text{in $C([0,T];L^r(\R^N))$ as $k\to\infty$}.
$$
\end{enumerate}
\end{theorem}

\begin{remark}
The classical Kolmogorov-Riesz compactness theorem requires the family of functions to be equibounded, that is, for all $t\in[0,T]$, 
$$
\sup_n\|u_n(\cdot,t)\|_{L^r(\R^N)}<\infty.
$$
However, in \cite{H-OHoMa17}, it has been pointed out that such a property follows from properties \eqref{sec2prop:equitight}--\eqref{sec2prop:equicontspace}.
\end{remark}

\begin{proof}
The fact that the sequence $x\mapsto \{u_n(x,t)\}_{n\in\N}$ satisfies \eqref{sec2prop:equitight}--\eqref{sec2prop:equicontspace} is, by the Kolmogorov-Riesz compactness theorem (cf. Theorem \ref{thm:KR}), equivalent with $\{x\mapsto \{u_n(x,t)\}_{n\in\N}\}$ being relatively compact in $L^r(\R^N)$ for all $t\in[0,T]$. Finally, since the sequence $t\mapsto \{u_n(x,t)\}_{n\in\N}$ satisfies \eqref{sec2prop:equiconttime}, the proof is completed by an application of the Arzel\`a-Ascoli compactness theorem (cf. Theorem \ref{thm:AA}). 
\end{proof}

\subsection{Consequences}
In our context, $\{u_n\}_{n\in\N}$ will be a sequence of,
e.g., numerical approximations of some function $u$ which could, e.g.,
be a distributional/very weak, entropy, energy, strong, mild, or
classical solution of \eqref{GPME}. The next properties therefore
relate the above compactness with results for the underlying
equation. When we say \emph{solve} and \emph{solution} in this section, we mean in one of the
senses mentioned above. The solution concepts in
\eqref{sec2prop:consistency} and \eqref{sec2prop:uniqueness} need to be the same. 

\begin{enumerate}[{\rm\bf (I)}]
\setcounter{enumi}{3}
\item \label{sec2prop:consistency}\textup{\bf (Consistent
  approximation)} Assume $\{u_n\}_{n\in\N}$ is a consistent
  approximation of
  problem (P), i.e.,
$$u_n\to u\quad\text{in $C([0,T];L^r(\R^N))$ as $n\to\infty$}\qquad
  \implies\qquad\text{$u$ solves (P).}
$$
\end{enumerate}

\begin{corollary}[Existence by consistent approximation]\label{metaExistence}
Assume $r\in[1,\infty)$. Let $\{u_n\}_{n\in\N}\subset C([0,T];L^r(\R^N))$ be a sequence satisfying properties \eqref{sec2prop:equitight}--\eqref{sec2prop:consistency}. Then there exist a solution $u\in C([0,T];L^r(\R^N))$ of \textup{(P)}.
\end{corollary}

\begin{proof}
By  \eqref{sec2prop:equitight}--\eqref{sec2prop:equiconttime}, any subsequence $\{u_{n_k}\}_{k\in\N}$ defined in Theorem \ref{metaCompactness} has a further subsequence  converging to some function $u\in C([0,T];L^r(\R^N))$. The consistency given by \eqref{sec2prop:consistency} tells us that the limit $u$ is a solution of (P).
\end{proof}

We end this discussion, by noting that full convergence of the sequence relies on uniqueness of the problem (P):

\begin{enumerate}[{\rm\bf (I)}]
\setcounter{enumi}{4}
\item \label{sec2prop:uniqueness}\textup{\bf (Uniqueness)} There is at most one solution $u\in C([0,T];L^r(\R^N))$ of (P). 
\end{enumerate}

\begin{proposition}[Convergence by uniqueness]\label{metaConvergence}
Assume $r\in[1,\infty)$. Let $\{u_n\}_{n\in\N}\subset C([0,T];L^r(\R^N))$ be a sequence satisfying properties \eqref{sec2prop:equitight}--\eqref{sec2prop:uniqueness}. Then 
$$
u_{n}\to u\quad\text{in \quad $C([0,T];L^r(\R^N))$\quad as \quad$n\to\infty$},
$$
where $u$ is the unique solution of \textup{(P)}.
\end{proposition}

\begin{proof}
Assume by contradiction that $\{u_n\}_{n\in\N}$ does not converge to $u$ in
$C_t(L_x^r)$. Then there is a subsequence and an $\veps>0$ such that
$\sup_{t\in[0,T]}\|u-u_{n_{k}}\|_{L^r(\R^N)}>\veps$ for every $k\in\N$.
By compactness (Theorem
\ref{metaCompactness}) there is a further subsequence  
$\{u_{n_{k_j}}\}_{j\in\N}$ converging in $C_t(L_x^r)$ to a function
$\tilde u$ which is a solution of (P) (Corollary
\ref{metaExistence}). However, uniqueness tells us that $u=\tilde u$, and we have a
contradiction. The whole sequence $\{u_n\}_{n\in\N}$ thus
converges to $u$ in $C_t(L_x^r)$.
\end{proof}

\subsection{Some variants} 

We will now discuss some variants of the above conditions which we will use in the paper, and also a comparison with other compactness results.

\subsubsection*{Equitight in space uniformly in time.} Now, we replace \eqref{sec2prop:equitight} and \eqref{sec2prop:equiconttime} by:
\begin{enumerate}[{\rm\bf (i)}]
\item \label{sec2prop:equitight'}\textup{\bf (Equitight in space uniformly in time)} For all $t\in[0,T]$,
$$
\lim_{R\to\infty}\sup_n\sup_{t\in[0,T]}\int_{|x|> R}|u_n(x,t)|^r\dd x=0.
$$
\setcounter{enumi}{2}
\item \label{sec2prop:equiconttime'}\textup{\bf (Equicontinuous in time)} For all $t,s\in[0,T]$ and all compact $K\subset \R^N$, there exists a modulus of continuity $\lambda_K$ such that
$$
\sup_n\int_K|u_n(\cdot,t)-u_n(\cdot,s)|\dd x\leq \lambda_K(|t-s|).
$$
\end{enumerate}

We immediately see that \eqref{sec2prop:equitight'} implies \eqref{sec2prop:equitight}. Moreover:

\begin{lemma}\label{FromLocalToGlobalEquicontinuity}
Assume $r\in[1,\infty)$. If the sequence $\{u_n\}_{n\in\N}\subset C([0,T];L^r(\R^N))$ satisfies \eqref{sec2prop:equitight'} and \eqref{sec2prop:equiconttime'}, then for all $\veps>0$ there exists $\eta>0$ such that for all $n$ and all $0\leq t\leq s\leq T$
\begin{equation}\label{sec2prop:equiconttimeNOMODULUS}
|t-s|<\eta\implies \int_{\R^N}|u_n(x,t)-u_n(x,s)|^r\dd x<\veps.
\end{equation}
\end{lemma}

\begin{remark}
The condition \eqref{sec2prop:equiconttimeNOMODULUS} is a weaker version of the original \eqref{sec2prop:equiconttime}. However, the second item of the Arzel\`a-Ascoli compactness theorem \ref{thm:AA} is still fulfilled, and hence, Theorem \ref{metaCompactness} still holds under the assumptions \eqref{sec2prop:equitight'}, \eqref{sec2prop:equicontspace}, and \eqref{sec2prop:equiconttime'}.
\end{remark}

\begin{proof}
Fix $\veps>0$ and the compact set $K$ such that $B_R(0)\subset K\subset \R^N$. Then
\begin{equation*}
\begin{split}
&\sup_n\int_{\R^N}|u_n(x,t)-u_n(x,s)|^r\dd x\\
&\leq \sup_n\int_{K}|u_n(x,t)-u_n(x,s)|^r\dd x+2\sup_n\sup_{t\in[0,T]}\int_{|x|> R}|u_n(x,t)|^r\dd x\\
\end{split}
\end{equation*}
By \eqref{sec2prop:equitight'}, we choose $R>0$ such that
$$
2\sup_n\sup_{t\in[0,T]}\int_{|x|> R}|u_n(x,t)|^r\dd x<\frac{\veps}{2},
$$
and then by \eqref{sec2prop:equiconttime'}, we take $\eta>0$ such that
$$
\sup_n\int_{K}|u_n(x,t)-u_n(x,s)|^r\dd x<\frac{\veps}{2}.
$$
The proof is complete.
\end{proof}

\subsubsection*{Variants of Arzel\`a-Ascoli and Kolmogorov-Riesz.}
Generalizations of the Arzel\`a-Ascoli compactness theorem to $L^p$-spaces can be found in Simon's well-cited paper \cite{Sim87}. There he discusses compactness of functions which are $L^p$ in time with values in some Banach space $B$. Sections 8 and 9 of that paper also contain what is commonly known as the Aubin-Lions lemma. Such an approach is different than what we do here, and is probably more suited in an energy-like setting. A nonlocal version is given by Theorem 3.1 in \cite{IgIgSt-Du15}.

Regarding Kolmogorov-Riesz on compact sets, the Helly compactness theorem can be used as a particular version in the $BV$-setting, see \cite[Section 6]{H-OHo10}.


\section{Main results}
\label{sec:MainResults}

\subsection{Assumptions and concept of solution}

Consider the following typical assumptions on \eqref{GPME}:
\begin{align}
&\varphi\in C_\textup{loc}^{0,\Hvar}(\R,\R), \text{$\Hvar\in(0,1]$, is nondecreasing}.
\tag{$\textup{A}_\varphi$}&
\label{phias}\\
&g\text{ is measurable and }\textstyle\int_0^T\big(\|g(\cdot,t)\|_{L^1(\R^N)}+\|g(\cdot,t)\|_{L^\infty(\R^N)}\big)\dd t<\infty.
\tag{$\textup{A}_g$}&
\label{gas}\\
&u_0\in L^1(\R^N)\cap L^{\infty}(\mathbb{R}^N).
\tag{$\textup{A}_{u_0}$}&
\label{u_0as}
\end{align}

In this paper, we restrict ourselves to $\Operator=L^\sigma+\Levy^\mu$, where 
\begin{align}\label{sigmaas}\tag{$\textup{A}_{\sigma}$} 
  L^\sigma=c \Delta:=c\sum_{i=1}^N\dell_{x_ix_i}^2\qquad\text{where $c\in\{0,1\}$},
  \end{align}
and $\Levy^\mu$ is a general symmetric L\'evy operator under the
usual assumption: 
\begin{align}\label{muas}\tag{$\textup{A}_{\mu}$} 
&\mu \text{ is a nonnegative symmetric Radon measure on
}\R^N\setminus\{0\}
\text{ satisfying}
\nonumber\\ 
&\int_{0<|z|\leq1}|z|^2\dd \mu(z)+\int_{|z|>1}\dd
\mu(z)<\infty\nonumber.
\end{align}

\begin{remark}
\begin{enumerate}[{\rm (a)}]
\item Assumption \eqref{gas} is the same as saying that 
$$
g\in L^1(0,T;L^1(\R^N)\cap L^\infty(\R^N))
$$
where the space is understood as iterated $L^{P}$-spaces in the sense of \cite{BePa61}. Of course, we have that $L^1(0,T;L^1(\R^N))=L^1(Q_T)$.
\item Without loss of generality, we can assume $\varphi(0) = 0$ (by adding constants to $\varphi$ in \eqref{GPME}). 
\item As solutions of \eqref{GPME} will be bounded, we can always assume that $\varphi\in C^{0,\Hvar}(\R)$ in \eqref{phias}. 
\end{enumerate}
\end{remark}

We will work with very weak solutions of \eqref{GPME}:

\begin{definition}[Very weak solutions]\label{def:VeryWeak} 
Let  $u_0\in L_\textup{loc}^1(\R^N)$ and $g\in  L_\textup{loc}^1(Q_T)$. Then $u\in  L_\textup{loc}^1(Q_T)$ is a \emph{very weak} solution of \eqref{GPME} if, for all $\psi\in C_\textup{c}^\infty(\R^N\times[0,T))$, $\varphi(u)\Operator[\psi]\in L^1(Q_T)$ and
$$
\iint_{Q_T} \big(u\dell_t\psi+\varphi(u)\Operator[\psi]+g\psi\big)\dd x \dd t+\int_{\R^N}u_0(x)\psi(x,0)\dd x=0.
$$
\end{definition}

Note that under the assumptions \eqref{phias}--\eqref{muas}, very weak
solutions of \eqref{GPME} are unique in $L^1(Q_T)\cap L^\infty(Q_T)$ by
\cite[Theorem 3.1]{DTEnJa17b}. By the same reference,
$$
\|u(\cdot,t)\|_{L^\infty(\R^N)}\leq \|u_0\|_{L^\infty(\R^N)}\qquad\text{and}\qquad u\in C([0,T];L_\textup{loc}^1(\R^N)),
$$
and hence, they are moreover in $C([0,T];L_\textup{loc}^r(\R^N))$ for all $r\in[1,\infty)$.

\subsection{Approximation through numerical method}

To define our numerical scheme, we need to introduce a discrete grid in $\R^N\times(0,T)$. Consider a sequence of numbers defining a nonuniform gird in time $\GridT:=\{t_j\}_{j=0}^{J}$ such that $0=t_0<t_1<\ldots <t_J=T$ and let $\J:=\{1,\ldots,J\}$. The time steps are then denoted by 
\begin{equation*}
\Delta t_j=t_j-t_{j-1}>0\textup{ for every $j\in \J$} \qquad \textup{and}  \qquad \Delta t=\max_{j\in \J}\{\Delta t_j\}.
\end{equation*}
Moreover, let the space step $h>0$, and consider the discrete subset of $\R^N$ given by 
\begin{equation*}
\Grid:= h\Z^N=\{x_\beta:=h \beta: \beta\in \Z^N\}.
\end{equation*}

We are now ready to define our numerical scheme. Since $u_0$ and $g$ do not necessarily have pointwise values, we set, for $R_h:=h(-\frac{1}{2},\frac{1}{2}]$,
\begin{equation}\label{discData}
U_\beta^0:=\frac{1}{h^N}\int_{x_\beta+R_h}u_0(x)\dd x\quad\text{and}\quad G_\beta^j:=\frac{1}{h^N\Delta t_j}\int_{t_{j-1}}^{t_j}\int_{x_\beta+R_h}g(x,\tau)\dd x \dd \tau.
\end{equation}
Then we seek a function $U:\Grid\times \GridT\to\R$ which solves
\begin{equation}\label{defNumSch}
\begin{split}
U_\beta^j=U_\beta^{j-1}+\Delta t_j\Levy^{h}[\varphi(U_{\cdot}^j)]_\beta+\Delta t_jG_\beta^j\quad\forall(\beta,j)\in(\Z^N,\J).
\end{split}
\end{equation}

\begin{remark}
Note that the above scheme is implicit in the diffusion term. This is done to ensure a simple theoretical analysis for merely H\"older continuous $\varphi$. If we choose an explicit scheme, we need to rely on a CFL-type stability condition which involves $|\varphi|_{C^{0,1}}^{-1}$, and this condition will blow up if $\varphi$ is not Lipschitz continuous. In the pure H\"older case, we would then need a further approximation of the nonlinearity. We refer the reader to \cite{DTEnJa18b, DTEnJa19} for details.
\end{remark}

The discrete variants of $\dell_t$ and $\Operator$ will now be discussed. For the partial derivative in time, we use the simple backward difference:
$$
\frac{\psi(t_j)-\psi(t_{j-1})}{\Delta t_j} \qquad\text{for all $\psi:\GridT\to\R$}.
$$
The finite-difference discretization of $\Delta$ is well-known and given by
\[
\Delta_h \psi(x)=\sum_{i=1}^{N}\frac{\psi(x+he_i)+\psi(x-he_i)-2\psi(x)}{h^2}.
\]
Recently, such approximation was shown to be in the class of L\'evy
operators for a certain finite measure \cite{DTEnJa17a, DTEnJa17b,
  DTEnJa18b, DTEnJa19}. We thus present a unified approach to
numerical discretizations for local and nonlocal operators. In fact,
by choosing the correct weights, we recover either local, nonlocal, or
combinations of both. Hence, consider the family of bounded,
symmetric, and monotone operators $\Levy^h: \ell^1(\Grid)\to \R$ given
by
\begin{equation*}
\Levy^h[\psi](x_\beta)=\Levy^h[\psi]_\beta=\sum_{\gamma\not=0} (\psi(x_\beta+z_\gamma)-\psi(x_\beta)) \omega_{\gamma,h} \quad\text{for all $\psi:\Grid\to\R$},
\end{equation*}
where $z_{-\gamma}=-z_\gamma\in \mathcal{G}_h$ and
$\sum_{\gamma\neq0}\omega_{\gamma,h}<\infty$. The sum is
over $0\neq \gamma\in \Z^N$ and the weights $\omega_{\gamma,h}$ such
that $\Levy^h$ converges to $\mathfrak{L}^{\sigma,\mu}$ in a suitable
sense, see \eqref{eq:OperatorConsistent} below. 
\begin{remark}
We can also write 
\[
\Levy^h[\psi]_\beta=\int_{\R^N}(\psi(x_\beta+z)-\psi(x_\beta))\dd \nu_h(z) \quad \textup{with} \quad \nu_h(z)= \sum_{\gamma\not=0} \delta_{z_\gamma}(z) w_{\gamma,h},
\]
where $\delta_{z_\gamma}$ is the dirac delta measure centered at $z_\gamma$ (cf. \cite{DTEnJa19}).
\end{remark}
We then need the following discrete versions of \eqref{muas}:
\begin{align}\label{Dmuas}\tag{$\textup{A}_{\omega_h}$} 
&\text{For all $h>0$ and $\gamma\neq0$, }\omega_{\gamma,h}=\omega_{-\gamma,h}\geq0\text{ and}
\nonumber\\ 
&\quad
\sup_{0<h<1}\bigg\{\sum_{h\leq |z_\gamma|\leq 1}|z_\gamma|^2\omega_{\gamma,h}+\sum_{|z_\gamma|>1}\omega_{\gamma,h}\bigg\}<\infty\nonumber.
\end{align}

\begin{remark}
The parallel of the above assumptions with \eqref{muas} might be easier to see in terms of $\nu_h$:
$$
\sup_{0<h<1}\bigg\{\int_{0<|z|\leq 1}|z|^2\dd \nu_h(z)
+\int_{|z|>1}\dd \nu_h(z)\bigg\}<\infty.
$$
In Appendix \ref{sec:DiscreteAndContinuousEstimatesNonlocal}, we also provide conditions for checking this assumption through operator consistency (see \eqref{eq:OperatorConsistent} below) and assumptions on the limit operator $\Operator$.
\end{remark}

Equitightness and convergence in $C([0,T];L^r(\R^N))$ with $r\in[1,\infty)$ of numerical schemes will be presented for a suitable interpolant which extends the discrete solutions to the whole space $\R^N\times(0,T)$. For that reason, let us define the piecewise \emph{constant} space interpolant $\Ux:\R^N\to\R$ of $U:\Grid \to \R$ as
\begin{equation*}
\Ux(x)=\sum_{x_\beta\in \Grid} \indik_{x_\beta+R_h}(x)U_\beta, \qquad\text{where $U_\beta=U(x_\beta)$,}
\end{equation*}
and the piecewise \emph{linear} in time and piecewise \emph{constant} in space space-time interpolant $\Uxt:\R^N\times[0,T]\to \R$ of $U:\Grid\times \GridT \to \R$ as
\begin{equation}\label{SpaceTimeInterpolant}
\begin{split}
\Uxt(x,t)&=\overline{U^0}(x)\mathbf{1}_{\{t_0\}}(t)\\
&\quad+\sum_{t_j\in\GridT\setminus{t_0}}  \indik_{(t_{j-1},t_j]}(t)\Big(\overline{U^{j-1}}(x)+\frac{t-t_{j-1}}{t_j-t_{j-1}}\big(\overline{U^j}(x)-\overline{U^{j-1}}(x)\big)\Big).
\end{split}
\end{equation}

\subsection{Equitightness}
Our results make use of a family of functions $\mathcal{X}$ of the form 
\begin{equation}\label{antiCutOff1}
\mathcal{X}_R(x):=\mathcal{X}\left(\frac{x}{R}\right)\qquad\text{with}\qquad R>0,
\end{equation}
where $\mathcal{X}$ is some fixed function such that
\begin{equation}\label{antiCutOff2}
\begin{split}
0\leq\mathcal{X}\in C_\textup{b}^\infty(\R^N)\qquad\text{and}\qquad \mathcal{X}(x)=\begin{cases}
0 & \text{if}\quad |x|\leq\frac{1}{2},\\
1 & \text{if}\quad |x|\geq 1.
\end{cases}
\end{split}
\end{equation}
Observe that $D^k\mathcal{X}\in C_{\textup{c}}(\R^N)$ for all $k\in \N$ and $1-\mathcal{X}\in C_{\textup{c}}(\R^N)$. Moreover:

\begin{lemma}\label{convDAntiCutOff}
Assume $k\in\N\cup\{0\}$. The function defined by \eqref{antiCutOff1}--\eqref{antiCutOff2} satisfies:
\begin{enumerate}[{\rm (a)}]
\item $\mathcal{X}_R\to0$ pointwise as $R\to\infty$.
\item If $\frac{N}{k}<p\leq\infty$, then $\|D^k\mathcal{X}_R\|_{L^p(\R^N)}=
R^{\frac{N}{p}-k}\|D^k\mathcal{X}\|_{L^p(\R^N)} \to0$ as $R\to\infty$.
\end{enumerate}
\end{lemma}

\begin{remark}\label{rem:convDAntiCutOff}
We note that $p=\infty$ gives the least restrictive convergence of $D^k\mathcal{X}_R\to0$ in $L^p(\R^N)$ as $R\to\infty$.
\end{remark}

From now on, we denote the space-time interpolant by $\Uxt_h$ to indicate its dependence on $h$.

\begin{theorem}[Equitightness estimate]\label{tailControlGPME2}
Assume \eqref{phias}, \eqref{gas}, \eqref{u_0as}, \eqref{Dmuas}, and $r\in[1,\infty)$. Let $p\in(1,1/(1-\Hvar)]$ and $q$ be such that $p^{-1}+q^{-1}=1$ (hence, $q\in[1/\Hvar,\infty)$), and let $U_\beta^j$ be an a.e.-solution of \eqref{defNumSch} corresponding to $U_\beta^0$ and $G_\beta^j$, and $\Uxt_h$ the corresponding interpolant. Then
\begin{equation*}
\begin{split}
&\sup_{t\in[0,T]}\int_{|x|>R}|\Uxt_h(x,t)|^r\dd x\\
&\leq M_{u_0,g}^{r-1}\Big(\|u_0\mathcal{X}_R\|_{L^1(\R^N)}+\|g\mathcal{X}_R\|_{L^1(0,T;L^1(\R^N))}+C_{u_0,g, \varphi}T\|\Levy^h[\mathcal{X}_R]\|_{L^p(\R^N)}\Big),
\end{split}
\end{equation*}
where 
\begin{equation*}
M_{u_0,g}:=\|u_0\|_{L^\infty(\R^N)}+\|g\|_{L^1(0,T;L^\infty(\R^N))}
\end{equation*}
and
\begin{equation*}
C_{\varphi,u_0,g}:=|\varphi|_{C^{0,\Hvar}}(M_{u_0,g})^{\Hvar-\frac{1}{q}}\big(\|u_0\|_{L^1(\R^N)}+\|g\|_{L^1(Q_T)}\big)^{\frac{1}{q}}.
\end{equation*}
\end{theorem}

\begin{remark}
We interpret $1/(\Hvar-1)=\infty$ when $\Hvar=1$.
\end{remark}

The proof, see Section \ref{sec:ProofOfEquitightness}, basically
consists of choosing $\mathcal{X}_R$ as a test function in Definition
\ref{def:VeryWeak}. Since $\varphi(u)\in L^q(Q_T)$ we need to rely on
$\Levy^h[\mathcal{X}_R]\to0$ in $L^p(\R^N)$
as $R\to\infty$. This puts a restriction on the class of operators, as
can be seen in the next result. To write down the result properly, we
will consider weights satisfying one of the following:
\begin{align}\label{Dmuas'}\tag{$\textup{A}_{\omega_h}'$} 
&\text{For all $h>0$ and $\gamma\neq0$, }\omega_{\gamma,h}=\omega_{-\gamma,h}\geq0\text{ and}
\nonumber\\ 
&
\sup_{0<h<1}\bigg\{\sum_{h\leq |z_\gamma|\leq 1}|z_\gamma|^2\omega_{\gamma,h}+\sum_{|z_\gamma|>1}|z_\gamma|\omega_{\gamma,h}\bigg\}<\infty\nonumber.
\end{align}
\begin{align}\label{DmuasFrac}\tag{$\textup{A}_{\omega_h}''$} 
&\text{In addition to \eqref{Dmuas}, also assume that, for $\alpha\in(0,2)$,}
\nonumber\\ 
&\sup_{R>1}\sup_{0<h<1}\bigg\{R^{\alpha-2}\sum_{1<|z_\gamma|\leq R}|z_\gamma|^2\omega_{\gamma,h}+R^\alpha\sum_{|z_\gamma|>R}\omega_{\gamma,h}\bigg\}\leq C.\nonumber
\end{align}

\begin{corollary}[Equitightness]\label{tailControlInLimitGPME2}
Assume \eqref{phias}, \eqref{gas}, \eqref{u_0as}, and $r\in[1,\infty)$. Let $U_\beta^j$ be an a.e.-solution of \eqref{defNumSch} corresponding to $U_\beta^0$ and $G_\beta^j$, and $\Uxt_h$ the corresponding interpolant. Then
\begin{equation*}
\begin{split}
\lim_{R\to\infty}\sup_{0<h<1}\sup_{t\in[0,T]}\int_{|x|> R}|\Uxt_h(x,t)|^r\dd x=0
\end{split}
\end{equation*}
if any of the following holds:
\begin{enumerate}[{\rm (i)}]
\item When $\Levy^h=\Delta_h$ (with weights satisfying \eqref{Dmuas}) and $\frac{(N-2)^+}{N}<\Hvar\leq 1$.
\item When $\Levy^h$ has weights satisfying \eqref{Dmuas'} and $\frac{N-1}{N}<\Hvar\leq 1$.
\item When $\Levy^h$ has weights satisfying \eqref{DmuasFrac} and $\frac{(N-\alpha)^+}{N}<\Hvar\leq1$.
\end{enumerate}
\end{corollary}

\begin{remark}
\begin{enumerate}[{\rm (a)}]
\item More generally, we also have equitightness when $\Levy^h$ has weights satisfying \eqref{Dmuas} and
$$
\sup_{0<h<1}\bigg\{R^{\frac{N}{p}-2}\sum_{1<|z_\gamma|\leq R}|z_\gamma|^2\omega_{\gamma,h}+R^{\frac{N}{p}}\sum_{|z_\gamma|>R}\omega_{\gamma,h}\bigg\}=o_R(1).
$$
\item Let us also provide the counterpart of \eqref{Dmuas'} and \eqref{DmuasFrac} for $\mu$:
\begin{align}\label{muas'}\tag{$\textup{A}_{\mu}'$} 
&\mu \text{ is a nonnegative symmetric Radon measure on }\R^N\setminus\{0\}\text{ satisfying}\nonumber\\ 
&\int_{0<|z|\leq1}|z|^2\dd \mu(z)+\int_{|z|>1}|z|\dd\mu(z)<\infty,\nonumber
\end{align}
and
\begin{align}\label{muasFrac}\tag{$\textup{A}_{\mu}''$} 
&\mu \text{ satisfies \eqref{muas} and, for $\alpha\in(0,2)$,}\nonumber\\ 
&\sup_{R>1}\bigg\{R^{\alpha-2}\int_{1<|z|\leq R}|z|^2\dd \mu(z)+R^\alpha\int_{|z|>R}\dd \mu(z)\bigg\}\leq C.\nonumber
\end{align}
The latter is satisfied for e.g.  $\mu$ such that \eqref{muas} holds and $\dd \mu(z)\leq
|z|^{-(N+\alpha)}\dd z$ for $|z|>1$ and some $\alpha\in(0,2)$. We again refer to Appendix  \ref{sec:DiscreteAndContinuousEstimatesNonlocal} for practical ways of checking \eqref{Dmuas'} and \eqref{DmuasFrac} through operator consistency (see \eqref{eq:OperatorConsistent} below) and assumptions on the limit operator $\Operator$.
\item Assumptions \eqref{muas'} and \eqref{muasFrac} are moment like conditions. They imply that $\mu\mathbf{1}_{|z|>1}$ has $s$ moments with $s=1$ or $s\in (0,\alpha)$, respectively. This is equivalent with the underlying L\'evy process having moments of order $s$, see e.g. Section 2.3 in \cite{ErJa21}.
\end{enumerate}
\end{remark}

\begin{proof}[Proof of Corollary \ref{tailControlInLimitGPME2}]
Let us summarize the direct computations done in Appendix \ref{sec:AuxRes}:
\begin{enumerate}[{\rm (i)}]
\item $\Delta_h\mathcal{X}_R\to0$ in $L^p(\R^N)$ as $R\to\infty$ if $p>\frac{N}{2}$ when $N\geq2$, and $p>1$ when $N=1$ (see Corollary \ref{wellPosednessLaplaceAntiCutOff}).
\item Under \eqref{Dmuas'}, $\Levy^h[\mathcal{X}_R]\to0$ in $L^p(\R^N)$ as $R\to\infty$ if $p>N$ (see Corollary \ref{wellPosednessLevyAntiCutOff}).
\item Under \eqref{DmuasFrac}, $\Levy^h[\mathcal{X}_R]\to0$ in $L^p(\R^N)$ as $R\to\infty$ if $p>N/\alpha$ when $N\geq 2$, $p>1/\alpha$ when $\alpha\in (0,1)$ and $N=1$, and $p>1$ when $\alpha\in[1,2)$ and $N=1$ (see Corollary \ref{wellPosednessLevyAntiCutOff} and Example \ref{ex:measure1}).
\end{enumerate}
Now, we can deduce a condition on $\Hvar$ in order to have convergence as $R\to\infty$, and then conclude by Theorem \ref{tailControlGPME2}. Choose $p=1/1-\Hvar$ (hence, $q=1/\Hvar$), and require that:
\begin{equation*}
\begin{cases}
1<\frac{1}{1-\Hvar} \quad&\qquad \text{in Case (i) for $N=1$,}\\
\frac{N}{2}<\frac{1}{1-\Hvar} \quad&\qquad \text{in Case (i) for $N\geq2$,}\\
N<\frac{1}{1-\Hvar} \quad&\qquad \text{in Case (ii) for $N\geq1$,}\\
\frac{N}{\alpha}<\frac{1}{1-\Hvar} \quad&\qquad \text{in Case (iii) for $N\geq2$,}\\
\frac{1}{\alpha}<\frac{1}{1-\Hvar} \quad&\qquad \text{in Case (iii) for $N=1$ and $\alpha\in(0,1)$,}\\
1<\frac{1}{1-\Hvar} \quad&\qquad \text{in Case (iii) for $N=1$ and $\alpha\in[1,2)$.}\\
\end{cases}
\end{equation*}
Hence, we get a restriction on the lower bound of $\Hvar$:
\begin{equation*}
\begin{cases}
0<\Hvar \quad&\qquad \text{in Case (i) for $N=1$,}\\
\frac{N-2}{N}<\Hvar \quad&\qquad \text{in Case (i) for $N\geq2$,}\\
\frac{N-1}{N}<\Hvar \quad&\qquad \text{in Case (ii) for $N\geq1$,}\\
\frac{N-\alpha}{N}<\Hvar \quad&\qquad \text{in Case (iii) for $N\geq2$,}\\
1-\alpha<\Hvar \quad&\qquad \text{in Case (iii) for $N=1$ and $\alpha\in(0,1)$,}\\
0<\Hvar \quad&\qquad \text{in Case (iii) for $N=1$ and $\alpha\in[1,2)$.}\\
\end{cases}
\end{equation*}
The proof is complete.
\end{proof}

\subsection{Compactness and convergence in $C([0,T];L^r(\R^N))$}\label{sec:CompactnessAndConvergenceResults}

We have so far checked property \eqref{sec2prop:equitight'} in Section \ref{sec:abstractCompConv}. However, properties \eqref{sec2prop:equicontspace} and \eqref{sec2prop:equiconttime'} follow from known results, see Lemma 4.5 and Lemma 4.6 in \cite{DTEnJa19}, respectively. Hence, as a consequence of the general presentation in Section \ref{sec:abstractCompConv} (and Lemma \ref{FromLocalToGlobalEquicontinuity}), we get: 

\begin{theorem}[Compactness]\label{thm:Compactness}
Assume \eqref{phias}--\eqref{muas}, and $r\in[1,\infty)$. Let $\{U_\beta^j\}_{h>0}$ be a sequence of a.e.-solutions of \eqref{defNumSch} and $\{\Uxt_h\}_{h>0}$ its corresponding sequence of interpolants. Then there exists a subsequence $\{\Uxt_{h_n}\}_{n\in\N}$ and a $u\in C([0,T];L^r(\R^N))$ such that and
$$
\Uxt_{h_n}\to u \qquad\text{in $C([0,T];L^r(\R^N))$ as $n\to\infty$,}
$$
if any of the following holds:
\begin{enumerate}[{\rm (i)}]
\item When $\Levy^h=\Delta_h$ (with weights satisfying \eqref{Dmuas}) and $\frac{(N-2)^+}{N}<\Hvar\leq1$.
\item When $\Levy^h$ has weights satisfying \eqref{Dmuas'} and $\frac{N-1}{N}<\Hvar\leq1$.
\item When $\Levy^h$ has weights satisfying \eqref{DmuasFrac} and $\frac{(N-\alpha)^+}{N}<\Hvar\leq1$.
\end{enumerate}
\end{theorem}

\begin{remark}
The lower bounds of $\Hvar$ are not surprising. In fact, equation \eqref{GPME} with  $\Operator=\Delta$ or $\Operator=-(-\Delta)^\frac{\alpha}{2}$, $\varphi(u)=|u|^{m-1}u$, and $g\equiv0$ will have solutions which become extinct in finite time (hence, there is no conservation of mass) when $m\in(0,\frac{(N-2)^+}{N})$ or $m\in(0,\frac{(N-\alpha)^+}{N})$ respectively  \cite{BeCr81a, DPQuRoVa12}. Those papers also prove that $\Hvar$ can reach the critical exponents $m=\frac{(N-2)^+}{N}$ or $m=\frac{(N-\alpha)^+}{N}$, but this is beyond the scope of the present paper. In effect, we include fast diffusion up to the critical exponents for both the Laplace operator and fractional Laplace like operators.
\end{remark}

In our concrete example of numerical approximations of \eqref{GPME}, Property \eqref{sec2prop:consistency} in Section \ref{sec:abstractCompConv} is a consequence of the obtained convergence $\Uxt_{h_n}\to u$ in $C([0,T];L^r(\R^N))$ as $n\to\infty$ and the fact that 
\begin{equation}\label{eq:OperatorConsistent}
\forall\,\psi\in C_\textup{c}^\infty(\R^N)\qquad  \Levy^h[\psi]\to\Operator[\psi]\qquad\text{in $L^1(\R^N)$,}
\end{equation}
see Lemma 4.8 in \cite{DTEnJa19} and its proof for further details.

\begin{corollary}[Existence]\label{cor:Existence}
Under the assumptions of Theorem \ref{thm:Compactness} and \eqref{eq:OperatorConsistent}, i.e., the schemes \eqref{defNumSch} are consistent, there exists a very weak solution of \eqref{GPME} in $L^1(Q_T)\cap L^\infty(Q_T)\cap C([0,T];L^r(\R^N))$.
\end{corollary}

The convergence of the whole sequence to the unique very weak solution of \eqref{GPME} follows as in the proof of Theorem \ref{metaConvergence}.

\begin{theorem}[Convergence]\label{thm:Convergence}
Under the assumptions of Corollary \ref{cor:Existence} and that the limit problem has a unique solution,
$$
\Uxt_{h}\to u \qquad \text{in $C([0,T];L^r(\R^N))$ as $h\to0^+$,}
$$
where $u$ is the unique very weak solution of \eqref{GPME} with data $u_0,g$.
\end{theorem}

\begin{remark}
Uniqueness of very weak solutions of \eqref{GPME} in $L^1(Q_T)\cap L^\infty(Q_T)$ can be found in \cite{DTEnJa17b}.
\end{remark}


\section{Proof of equitightness}
\label{sec:ProofOfEquitightness}

\subsection{Preliminaries}
Let us start by recalling some properties of $U_\beta^j$.

\begin{proposition}[{{\cite{DTEnJa19}}}]\label{preResNumSch}
Assume \eqref{phias}, \eqref{gas}, \eqref{u_0as}, \eqref{Dmuas}.
\begin{enumerate}[{\rm (a)}]
\item There exists a unique a.e.-solution $U_\beta^j$ of \eqref{defNumSch} such that
  \begin{equation*}
  \sum_{j\in\J}\sum_{\beta\neq0}|U_\beta^j|<\infty.
  \end{equation*}
\item If $U_\beta^j,\hat{U}_\beta^j$ are a.e.-solutions of \eqref{defNumSch} with data $U_\beta^0,G_\beta^j$ and $\hat{U}_\beta^0,\hat{G}_\beta^j$ respectively, then:
	\begin{enumerate}[{\rm (i)}]
	\item \textup{(Monotonicity)} If $U_\beta^0\leq \hat{U}_\beta^0$ and $G_\beta^j\leq \hat{G}_\beta^j$, then $U_\beta^j\leq \hat{U}_\beta^j$.
	\item \textup{($L^1(\Z^N)$-contraction)} 
	$$
	\sum_{\beta\neq0} (U_\beta^j-\hat{U}_\beta^j)^+
\leq \sum_{\beta\neq0} (U_\beta^0-\hat{U}_\beta^0)^+ + \sum_{l=1}^j\bigg(  \Delta t_{l}\sum_{\beta\neq0} (G_\beta^l-\hat{G}^j_\beta)^+\bigg).
	$$
		\item \textup{(Conservative)} If $\Hvar=1$, then
	$$
	\sum_{\beta\neq0}U_\beta^j=\sum_{\beta\neq0}U_\beta^0+\sum_{l=1}^j\bigg(\Delta t_l\sum_{\beta\neq0}G_\beta^l\bigg).
	$$
	\item \textup{($L^1(\Z^N)$-stability)} $\sum_{\beta\neq0}|U_\beta^j|\leq \sum_{\beta\neq0}|U_\beta^0|+\sum_{l=1}^j\big(\Delta t_l\sum_{\beta\neq0}|G_\beta^l|\big).$
	\item \textup{($L^\infty(\Z^N)$-stability)} $\sup_{\beta\neq0}|U_\beta^j|\leq \sup_{\beta\neq0}|U_\beta^0|+\sum_{l=1}^j\big(\Delta t_l\sup_{\beta\neq0}|G_\beta^l|\big).$
	\item \textup{(Equicontinuity in time)} For all compact sets $K\subset \R^N$, and some modulus of continuity $\lambda_K$,
	$$
	h^N\sum_{x_\beta\in \mathcal{G}_h\cap K}|U_\beta^j-U_\beta^{j-k}|\leq \lambda_K(t_j-t_{j-k})+|K|\int_{t_{j-k}}^{t_j}\|g(\cdot,\tau)\|_{L^\infty(\R^N)}\dd \tau.
	$$
	\end{enumerate}
\end{enumerate}
\end{proposition}

\begin{remark}
These properties can of course be reformulated in terms of the the space interpolant $\overline{U^j}$, e.g., the $L^1(\R^N)$-contraction becomes
$$
\int_{\R^N}\Big(\overline{U^j}-\overline{\hat{U}^j}\Big)^+\dd x\leq \int_{\R^N}(u_0-\hat{u}_0)^++\int_0^{t_j}\int_{\R^N}(g-\hat{g})^+\dd x\dd \tau.
$$
\end{remark}

Since our discrete operator $\Levy^h$ have weights and stencils not depending on the spatial variable $x$, we transform our space-time scheme \eqref{defNumSch} into the following scheme only discretized in time:
\begin{equation}\label{defNumSchTime}
\begin{split}
U^j(x)&=U^{j-1}(x)+\Delta t_j\Levy^{h}[\varphi(U^j)](x)+\Delta t_jG^j(x)\quad\forall j\in\J.
\end{split}
\end{equation}
Indeed:

\begin{lemma}\label{lem:equivalenceNumSchNumSchTime} 
Assume \eqref{phias}, \eqref{gas}, \eqref{u_0as}, \eqref{Dmuas}, and let
$U_\beta^0$, $G_\beta^j$ be defined by \eqref{discData} and 
$\overline{U^0}$, $\overline{G^j}$ be the same functions interpolated in space, and let
$\Levy^h$ have $x$-independent stencils.
If $U_\beta^j$ is a solution of \eqref{defNumSch} with data $U^0_\beta$ and $G_\beta^j$, then $\overline{U^j}(x)$ is a piecewise constant solution of \eqref{defNumSchTime} with data $\overline{U^0}$ and $\overline{G^j}$. 
\end{lemma}

The proof follows by uniqueness of $U_\beta^j$ and translation invariance in $x$ of \eqref{defNumSch}, see Proposition 2.10 (b) in \cite{DTEnJa19}. It is more convenient to work with solutions of \eqref{defNumSchTime} since they are defined for all spatial points in $\R^N$. Finally, any result will then be inherited by $\Uxt$. To ease the notation, $U^j$ will denote solutions of \eqref{defNumSchTime} and its time interpolant as $\widetilde{U}$.

\subsection{Proof of the equitightness estimates}

We can write \eqref{defNumSchTime} as
\begin{equation*}
\begin{split}
U^j(x)-\Delta t_j\Levy^{h}[\varphi(U^j)](x)=U^{j-1}(x)+\Delta t_jG^j(x)\quad\forall j\in\J.
\end{split}
\end{equation*}
That is, by following \cite{DTEnJa19}, we define
\begin{equation*}
\Ti[\rho](x):=w(x)\qquad \text{for $x\in\R^N$},
\end{equation*} 
where $w$ is the solution of the nonlinear elliptic equation
\begin{equation}\label{EP}\tag{EP}
w(x)-\Delta t_j\Levy^{h}[\varphi(w)](x)=\rho(x)\qquad \text{for $x\in\R^N$}.
\end{equation}
Then the solution of the scheme \eqref{defNumSchTime} can be written as
\begin{equation*}
U^j(x)=\Ti\big[U^{j-1}+\Delta t_jG^j\big](x)\qquad \text{for $x\in\R^N$}.
\end{equation*}

\begin{proposition}[{{\cite[Section 4]{DTEnJa19}}}]\label{prop:AprioriTi}
Assume \eqref{phias}. If $\rho\in L^1(\R^N)\cap L^\infty(\R^N)$, then there exists a unique a.e.-solution $\Ti[\rho]=w\in L^1(\R^N)\cap L^\infty(\R^N)$ of \eqref{EP}. Moreover, for $\rho,\hat{\rho}\in L^1(\R^N)\cap L^\infty(\R^N)$:
\begin{enumerate}[{\rm (a)}]
\item \textup{(Comparison)} If $\rho\leq \hat{\rho}$ a.e., then $\Ti[\rho]\leq
\Ti[\hat{\rho}]$ a.e.
\item  \textup{($L^1(\R^N)$-contraction)}
$\int_{\R^N}(\Ti[\rho](x)-\Ti[\hat{\rho}](x))^+\dd x\leq\int_{\R^N}(\rho(x)-\hat{\rho}(x))^+\dd x$.
\item \textup{($L^p(\R^N)$-bound)} For $p\in[1,\infty]$, $\|\Ti[\rho]\|_{L^p(\R^N)}\leq \|\rho\|_{L^\infty(\R^N)}^{(p-1)/p}\|\rho\|_{L^1(\R^N)}^{1/p}$.
\end{enumerate}
\end{proposition}

We then deduce a tail estimate for $\Ti[\rho]$.

\begin{lemma}\label{lem:TailControlTi}
Under the assumptions of Proposition \ref{prop:AprioriTi},
\begin{equation}\label{tailControlTi}
\begin{split}
\int_{\R^N}|\Ti[\rho]|\mathcal{X}_R\dd x\leq\int_{\R^N}|\rho|\mathcal{X}_R\dd x+\Delta t_j\int_{\R^N}|\varphi(\Ti[\rho])|\Levy^h[\mathcal{X}_R]\dd x.
\end{split}
\end{equation}
\end{lemma}

\begin{proof}
Multiply \eqref{EP} by $\sgn(w)$, and note that $w\sgn(w)=|w|$, $\rho\sgn(w)\leq |\rho|$, and $\big(\varphi(w(x+y))-\varphi(w(x))\big)\sgn(w(x))\leq \big(|\varphi(w(x+y))|-|\varphi(w(x))|\big)$. The latter is due to the fact that $\varphi$ is nondecreasing. Then
$$
\Levy^h[\varphi(w)](x)\sgn(w(x))\leq \Levy^h[|\varphi(w)|](x).
$$
Combining the above yields the pointwise estimate
$$
|w|\leq|\rho|+\Delta t_j\Levy^h[|\varphi(w)|].
$$
Then we multiply by $\mathcal{X}_R$ defined in \eqref{antiCutOff1}--\eqref{antiCutOff2} and integrate over $\R^N$ to obtain
\begin{equation*}
\begin{split}
\int_{\R^N}|w|\mathcal{X}_R\dd x\leq\int_{\R^N}|\rho|\mathcal{X}_R\dd x+\Delta t_j\int_{\R^N}\Levy^h[|\varphi(w)|]\mathcal{X}_R\dd x.
\end{split}
\end{equation*}
By Tonelli's lemma, the properties $z_{-\gamma}=-z_{\gamma}$, and the changes of variables $x\mapsto x-z_{\gamma}$ and $\gamma\mapsto-\gamma$,
\begin{equation*}
\begin{split}
\int_{\R^N}\Levy^h[|\varphi(w)|](x)\mathcal{X}_R(x)\dd x&=\int_{\R^N}\Levy^h[\mathcal{X}_R](x)|\varphi(w(x))|\dd x.\\
\end{split}
\end{equation*} 
The proof is then complete.
\end{proof}

We are then ready to prove our equitightness estimate.

\begin{proof}[Proof of Theorem \ref{tailControlGPME2}]
\noindent \textbf{1)} {\em $L^1(\R^N)$-equitightness estimate for $U^{j}$.} We simply use \eqref{tailControlTi} to get
\begin{equation*}
\begin{split}
&\int_{\R^N}|U^j|\mathcal{X}_R\dd x=\int_{\R^N}|\Ti[U^{j-1}+\Delta t_jG^j]|\mathcal{X}_R\dd x\\
&\leq\int_{\R^N}|U^{j-1}+\Delta t_jG^j|\mathcal{X}_R\dd x+\Delta t_j\int_{\R^N}|\varphi(\Ti[U^{j-1}+\Delta t_jG^j])||\Levy^h[\mathcal{X}_R]|\dd x\\
&\leq \int_{\R^N}|U^{j-1}|\mathcal{X}_R\dd x+\Delta
  t_j\int_{\R^N}|G^j|\mathcal{X}_R\dd x+\Delta
  t_j|\varphi|_{C^{0,\Hvar}}\int_{\R^N}|U^j|^{\Hvar}|\Levy^h[\mathcal{X}_R]|\dd x.
\end{split}
\end{equation*}
After iterating $j$ down to zero, the above estimate yields
\begin{equation*}
\begin{split}
&\int_{\R^N}\big|U^j\big|\mathcal{X}_R\dd x\\
&\leq\int_{\R^N}|u_0|\mathcal{X}_R\dd x+\sum_{l=1}^j\Delta t_l\int_{\R^N}|G^l|\mathcal{X}_R\dd x+|\varphi|_{C^{0,\Hvar}}\sum_{l=1}^j\Delta t_l\int_{\R^N}|U^l|^{\Hvar}|\Levy^h[\mathcal{X}_R]|\dd x\\
&\leq\int_{\R^N}|u_0|\mathcal{X}_R\dd x+\int_0^{t_j}\int_{\R^N}|g(x,\tau)|\mathcal{X}_R\dd x\dd \tau+|\varphi|_{C^{0,\Hvar}}\sum_{l=1}^j\Delta t_l\int_{\R^N}|U^l|^{\Hvar}|\Levy^h[\mathcal{X}_R]|\dd x.
\end{split}
\end{equation*}

Under the assumptions on $p,q$ and by Proposition \ref{preResNumSch}(b)(iv)--(v) and the H\"older inequality, we get
$$
\int_{\R^N}|U^l|^\Hvar|\Levy^h[\mathcal{X}_R]|\dd x\leq (M_{u_0,g})^{\Hvar-\frac{1}{q}}\big(\|u_0\|_{L^1(\R^N)}+\|g\|_{L^1(Q_T)}\big)^{\frac{1}{q}}\|\Levy^h[\mathcal{X}_R]\|_{L^p(\R^N)}.
$$
Hence,
$$
\int_{\R^N}\big|U^j\big|\mathcal{X}_R\dd x\leq \|u_0\mathcal{X}_R\|_{L^1(\R^N)}+\|g\mathcal{X}_R\|_{L^1(0,T;L^1(\R^N))}+C_{\varphi,u_0,g}T\|\Levy^h[\mathcal{X}_R]\|_{L^p(\R^N)}.
$$
Note that the estimate is independent of $j$.

\medskip
\noindent \textbf{2)} {\em $L^1(\R^N)$-equitightness estimate for $\widetilde{U}$}. Assume $t\in(t_{j-1},t_j]$ for some $j\in\mathbb{J}$. By \eqref{SpaceTimeInterpolant}, 
\begin{equation}\label{eq:SpaceTimeInterpolant2}
\widetilde{U}(t)=U^{j-1}+\frac{t-t_{j-1}}{t_j-t_{j-1}}\Big(U^{j}-U^{j-1}\Big)=\frac{t-t_{j-1}}{\Delta t_j}U^j+\frac{t_j-t}{\Delta t_j}U^{j-1},
\end{equation}
and hence, Step 1) then yields
\begin{equation*}
\begin{split}
&\int_{\R^N}|\widetilde{U}(t)|\mathcal{X}_R\dd x\\
&\leq \frac{t-t_{j-1}}{\Delta t_j}\int_{\R^N}|U^j|\mathcal{X}_R\dd x+\frac{t_j-t}{\Delta t_j}\int_{\R^N}|U^{j-1}|\mathcal{X}_R\dd x\\
&\leq  \|u_0\mathcal{X}_R\|_{L^1(\R^N)}+\|g\mathcal{X}_R\|_{L^1(0,T;L^1(\R^N))}+C_{\varphi, u_0,g}T\|\Levy^h[\mathcal{X}_R]\|_{L^p(\R^N)}.
\end{split}
\end{equation*}

\medskip
\noindent \textbf{3)} {\em $L^r$-equitightness estimate for $\widetilde{U}$.} By Proposition \ref{preResNumSch}(b)(v) and \eqref{eq:SpaceTimeInterpolant2},
$$
\|\widetilde{U}(t)\|_{L^\infty(\R^N)}\leq \|u_0\|_{L^\infty(\R^N)}+\int_0^T\|g(\cdot,\tau)\|_{L^\infty(\R^N)}\dd\tau=M_{u_0,g}.
$$ 
Since $\mathbf{1}_{|x|> R}\leq \mathcal{X}_R(x)$,
\begin{equation*}
\begin{split}
\int_{|x|>R}|\widetilde{U}(t)|^r\dd x\leq \int_{\R^N}|\widetilde{U}(t)|^r\mathcal{X}_R\dd x\leq M_{u_0,g}^{r-1}\int_{\R^N}|\widetilde{U}(t)|\mathcal{X}_R\dd x,
\end{split}
\end{equation*}
and then we inherit the estimate from Step 2).
\end{proof}


\section{Extensions and comments}
\label{sec:ExtensionsAndComments}

\subsection{Improved limit regularity in the Lipschitz case}
\label{sec:ImprovedLimitRegularityLip}

When $\varphi$ is Lipschitz (\eqref{phias} with $\Hvar=1$),
compactness in $C([0,T];L^r_{\textup{loc}}({\RN}))$ is enough to show
that the limit point (a very weak solution 
of \eqref{GPME}) belongs to $C([0,T];L^r(\R^N))$. By previous results,
this can be obtained without equitightness and full $L^p(\R^N)$-convergence. Hence, it holds for all
operators $\Operator$ on the form
\eqref{defbothOp}--\eqref{defLocalOp}--\eqref{deflevy} under
\eqref{sigmaas} and \eqref{muas}. To also have convergence in
$C([0,T];L^r(\R^N))$, equitightness and then additional restrictions on
$\Operator$ are needed as in Theorem
\ref{thm:Compactness}.

The key estimate to show such a regularity is that we can obtain global in space equicontinuity in time for $\Uxt_{h}$. To do so, we introduce the following moduli of continuity:

\begin{equation}\label{moduliu_0g}
\lambda_{u_0}(\zeta):=\sup_{|\xi|\leq \zeta}\|u_0-u_0(\cdot+\xi)\|_{L^1(\R^N)},\, 
\lambda_g(\zeta):=\sup_{|\xi|\leq \zeta}\|g-g(\cdot+\xi,\cdot+0)\|_{L^1(Q_T)}.
\end{equation}

\begin{proposition}[Global in space equicontinuity in time]\label{equicontinuousTimeEdiff}
Assume \eqref{phias} with $\Hvar=1$, \eqref{gas}, \eqref{u_0as}, \eqref{Dmuas}, and $r\in[1,\infty)$. Let $U_\beta^j$ be an a.e.-solution of \eqref{defNumSch} corresponding to $U_\beta^0$ and $G_\beta^j$, and $\Uxt_h$ the corresponding interpolant. Then, for all $t,s\in[0,T]$,
\begin{equation*}
\begin{split}
\sup_{0<h<1}\|\Uxt_h(\cdot,t)-\Uxt_h(\cdot,s)\|_{L^r(\R^N)}\leq\Lambda(|t-s|)+\int_{s}^{t}\|g(\cdot,\tau)\|_{L^r(\R^N)}\dd \tau,
\end{split}
\end{equation*}
where 
\begin{equation*}
\Lambda(\zeta):=C_{u_0,g,\varphi}\inf_{\delta>0}\Big(\lambda_{u_0}(\delta)+\lambda_g(\delta)+\zeta\big(\delta^{-2}+1\big)\Big).
\end{equation*}
\end{proposition}

We postpone the proof to the end of this section. 

\begin{proposition}
Assume \eqref{phias} with $\Hvar=1$, \eqref{gas}, \eqref{u_0as}, \eqref{Dmuas}, and $r\in[1,\infty)$. Let $U_\beta^j$ be an a.e.-solution of \eqref{defNumSch} corresponding to $U_\beta^0$ and $G_\beta^j$, and $\Uxt_h$ the corresponding interpolant. Then
$$
\Uxt_{h}\to u \qquad \text{in $C([0,T];L_\textup{loc}^r(\R^N))$ as $h\to0^+$,}
$$
where $u$ is the unique very weak solution of \eqref{GPME}. Moreover, $u\in L^\infty(Q_T)\cap C([0,T];L^r(\R^N))$.
\end{proposition}

\begin{proof}
According to Theorem 4.7 in \cite{DTEnJa19}, we have that
$$
\Uxt_{h}\to u \qquad \text{in $C([0,T];L_\textup{loc}^r(\R^N))$ as $h\to0^+$,}
$$
and $u\in L^1(Q_T)\cap L^\infty(Q_T)\cap C([0,T];L_\textup{loc}^r(\R^N))$. Note that, for all $t,s\in[0,T]$, we use Proposition \ref{equicontinuousTimeEdiff} to obtain
\begin{equation*}
\begin{split}
&\big\|\big(u(\cdot,t)-u(\cdot,s)\big)\mathbf{1}_{|\cdot|\leq R}\big\|_{L^r(\R^N)}\\
&=\big\|\big(u(\cdot,t)-\Uxt_{h}(\cdot,t)\big)\mathbf{1}_{|\cdot|\leq R}\big\|_{L^r(\R^N)}+\big\|\big(\Uxt_{h}(\cdot,t)-\Uxt_{h}(\cdot,s)\big)\mathbf{1}_{|\cdot|\leq R}\big\|_{L^r(\R^N)}\\
&\quad+\big\|\big(\Uxt_{h}(\cdot,s)-u(\cdot,s)\big)\mathbf{1}_{|x|\leq R}\big\|_{L^r(\R^N)}\\
&\leq 2\big\|\big(u-\Uxt_{h}\big)\mathbf{1}_{|\cdot|\leq R}\big\|_{C([0,T];L^r(\R^N))}+\sup_{0<h<1}\big\|\Uxt_{h}(\cdot,t)-\Uxt_{h}(\cdot,s)\big\|_{L^r(\R^N)}\\
&\leq 2\big\|\big(u-\Uxt_{h}\big)\mathbf{1}_{|\cdot|\leq R}\big\|_{C([0,T];L^r(\R^N))}+\Lambda(|t-s|)+\int_{s}^{t}\|g(\cdot,\tau)\|_{L^r(\R^N)}\dd \tau.
\end{split}
\end{equation*}
Now, send $h\to0^+$, and then, by the Fatou lemma, $R\to\infty$. This gives
$$
\|u(\cdot,t)-u(\cdot,s)\big\|_{L^r(\R^N)}\leq \Lambda(|t-s|)+\int_{s}^{t}\|g(\cdot,\tau)\|_{L^r(\R^N)}\dd \tau\to0\qquad\text{as $|t-s|\to0$,}
$$
and the desired regularity is obtained.
\end{proof}

\begin{remark}
As a byproduct, we inherit an equicontinuity in time estimate for $u$.
\end{remark}

\begin{proof}[Proof of Proposition \ref{equicontinuousTimeEdiff}]
To simplify we  write $U^j,\widetilde{U}$ for $\overline{U^j}, \Uxt$. The proof is divided into two steps.

\medskip
\noindent\textbf{1)} {\em Deducing the result for a smooth $U^j$.} Define $U^{j}_\delta(x):=U^j\ast\rho_\delta(x)$ (and similarly for $G^{j}_\delta$). Now, multiply \eqref{defNumSchTime} by $\rho_\delta(x-y)$ and integrate over $\R^N$ (in $y$) to obtain
\begin{equation*}
\begin{split}
U^j_\delta(x)-U^{j-1}_\delta(x)&=\Delta t_j\Levy^{h}[\varphi(U^j)]\ast\rho_\delta(x)+\Delta t_jG^j_\delta(x).
\end{split}
\end{equation*}
By a simple change of variables and Fubini's theorem,
\begin{equation*}
\Levy^{h}[\varphi(U^j)]\ast\rho_\delta(x)=\Levy^{h}[\rho_\delta]\ast\varphi(U^j)(x).
\end{equation*}
We thus have that
\begin{equation}\label{preTimeReg}
\begin{split}
\|U^j_\delta-U^{j-1}_\delta\|_{L^{r}(\R^N)}\leq \Delta t_j\|\Levy^{h}[\rho_\delta]\ast\varphi(U^j)\|_{L^r(\R^N)}+\Delta t_j\|G^j_\delta\|_{L^r(\R^N)}.
\end{split}
\end{equation}
The Young inequality for convolutions and \eqref{phias} (with $\Hvar=1$) give
\begin{equation*}
\begin{split}
\|\Levy^{h}[\rho_\delta]\ast\varphi(U^j)\|_{L^r(\R^N)}&\leq \|\Levy^{h}[\rho_\delta]\|_{L^1(\R^N)}\|\varphi(U^j)\|_{L^r(\R^N)}\\
&\leq |\varphi|_{C^{0,1}}\|\Levy^{h}[\rho_\delta]\|_{L^1(\R^N)}\|U^j\|_{L^{r}(\R^N)}
\end{split}
\end{equation*}
and
\begin{equation*}
\begin{split}
\|\overline{G^j}_\delta\|_{L^r(\R^N)}&\leq\|\rho_\delta\|_{L^{1}(\R^N)}\|\overline{G^j}\|_{L^{r}(\R^N)}.
\end{split}
\end{equation*}

Direct computations and Lemma \ref{lem:estimateDiscNonlocal}(a) yield $\|\rho_\delta\|_{L^1(\R^N)}=1$ and
\begin{equation*}
\begin{split}
&\|\Levy^h[\rho_\delta]\|_{L^1(\R^N)}\\
&\leq \|D^2\rho_\delta\|_{L^1(\R^N)}\sup_{0<h<1}\bigg\{\sum_{h<|z_\gamma|\leq 1}|z_\gamma|^2\omega_{\gamma,h}\bigg\}+2\|\rho_\delta\|_{L
^1(\R^N)}\sup_{0<h<1}\bigg\{\sum_{|z_\gamma|>1}\omega_{\gamma,h}\bigg\},
\end{split}
\end{equation*}
where
$$
\|D^2\rho_\delta\|_{L^1(\R^N)}=\frac{1}{\delta^{2}}\|D^2\rho\|_{L^1(\R^N)}.
$$
Now, combine the above estimates together with Proposition \ref{preResNumSch} (b)(iv) and (b)(v) to get
\begin{equation*}
\begin{split}
\|\Levy^{h}[\rho_\delta]\ast\varphi(U^j)\|_{L^r(\R^N)}\leq |\varphi|_{C^{0,1}}\|\Levy^{h}[\rho_\delta]\|_{L^1(\R^N)}\|U^j\|_{L^{r}(\R^N)}\leq C_{u_0,g,\varphi}\Big(\delta^{-2}+1\Big)
\end{split}
\end{equation*}
where
\begin{equation*}
\begin{split}
C_{u_0,g,\varphi}&:=|\varphi|_{C^{0,1}}\Big(\|D^2\rho\|_{L^{1}(\R^N)}+2\Big)\sup_{0<h<1}\bigg\{\sum_{h<|z_\gamma|\leq 1}|z_\gamma|^2\omega_{\gamma,h}+\sum_{|z_\gamma|>1}\omega_{\gamma,h}\bigg\}\\
&\qquad\times M_{u_0,g}^{\frac{r-1}{r}}\left(\|u_0\|_{L^1(\R^N)}+\int_0^{T}\|g(\cdot,\tau)\|_{L^1(\R^N)}\dd \tau\right)^{\frac{1}{r}}.
\end{split}
\end{equation*}

\medskip
\noindent\textbf{2)} {\em Inheriting the result for $\widetilde{U}$.} Let $t\in(t_{j-1},t_j]$ and $s\in(t_{j-k-1},t_{j-k}]$, then we have, for the corresponding interpolant in time,
\begin{equation*}
\begin{split}
\|\widetilde{U}(\cdot,t)-\widetilde{U}_\delta(\cdot,s)\|_{L^{r}(\R^N)}&\leq \big\|\big(\widetilde{U}_\delta(\cdot,t)-\widetilde{U}_\delta(\cdot,t_{j-1})\big)+\big(\widetilde{U}_\delta(\cdot,t_{j-k})-\widetilde{U}_\delta(\cdot,s)\big)\big\|_{L^{r}(\R^N)}\\
&\quad+\sum_{l=j-k+1}^{j-1}\|\widetilde{U}_\delta(\cdot,t_l)-\widetilde{U}_\delta(\cdot,t_{l-1})\|_{L^{r}(\R^N)}\\
&=\big\|\big(\widetilde{U}_\delta(\cdot,t)-U^{j-1}_\delta\big)+\big(U^{j-k}_\delta-\widetilde{U}_\delta(\cdot,s)\big)\big\|_{L^{r}(\R^N)}\\
&\quad+\sum_{l=j-k+1}^{j-1}\|U^{l}_\delta-U^{l-1}_\delta\|_{L^{r}(\R^N)}.\\
\end{split}
\end{equation*}
Since we use piecewise linear interpolation in time,
$$
\widetilde{U}_\delta(x,t)=U^{j-1}_\delta(x)+\frac{t-t_{j-1}}{t_j-t_{j-1}}\big(U^j_\delta(x)-U^{j-1}_\delta(x)\big)
$$
and
$$
\widetilde{U}_\delta(x,s)=U^{j-k-1}_\delta(x)+\frac{s-t_{j-k-1}}{t_{j-k}-t_{j-k-1}}\big(U^{j-k}_\delta(x)-U^{j-k-1}_\delta(x)\big).
$$
Hence,
\begin{equation*}
\begin{split}
&\big\|\big(\widetilde{U}_\delta(\cdot,t)-U^{j-1}_\delta\big)+\big(U^{j-k}_\delta-\Uxt_\delta(\cdot,s)\big)\big\|_{L^{r}(\R^N)}\\
&\leq \frac{t-t_{j-1}}{\Delta t_j}\|U^j_\delta-U^{j-1}_\delta\|_{L^{r}(\R^N)}+\frac{t_{j-k}-s}{\Delta t_{j-k}}\|U^{j-k}_\delta-U^{j-k-1}_\delta\|_{L^{r}(\R^N)}.
\end{split}
\end{equation*}
By iterating \eqref{preTimeReg}, using Tonelli's lemma, equation \eqref{moduliu_0g}, and Proposition \ref{preResNumSch}(b)(iii), we obtain
\begin{equation*}
\begin{split}
&\|\widetilde{U}(\cdot,t)-\widetilde{U}(\cdot,s)\|_{L^{r}(\R^N)}\\
&\leq \|\widetilde{U}(\cdot,t)-\widetilde{U}_\delta(\cdot,t)\|_{L^{r}(\R^N)}+\|\widetilde{U}_\delta(\cdot,t)-\widetilde{U}_\delta(\cdot,s)\|_{L^{r}(\R^N)}+\|\widetilde{U}_\delta(\cdot,s)-\widetilde{U}(\cdot,s)\|_{L^{r}(\R^N)}\\
&\leq 2^{1+\frac{r-1}{r}}M_{u_0,g}^{\frac{r-1}{r}}\big(\lambda_{u_0}(\delta)+\lambda_g(\delta)\big)+\frac{t-t_{j-1}}{\Delta t_j}\|U^j_\delta-U^{j-1}_\delta\|_{L^{r}(\R^N)}\\
&\quad+\sum_{l=j-k+1}^{j-1}\|U^{l}_\delta-U^{l-1}_\delta\|_{L^{r}(\R^N)}+\frac{t_{j-k}-s}{\Delta t_{j-k}}\|U^{j-k}_\delta-U^{j-k-1}_\delta\big\|_{L^{r}(\R^N)}\\
&\leq 2^{1+\frac{r-1}{r}}M_{u_0,g}^{\frac{r-1}{r}}\big(\lambda_{u_0}(\delta)+\lambda_g(\delta)\big)+C_{u_0,g,\varphi,\omega,\rho}|t-s|\Big(\delta^{-2}+1\Big)\\
&\quad +\int_{s}^{t}\|g(\cdot,\tau)\|_{L^r(\R^N)}\dd \tau.\\
\end{split}
\end{equation*}
The proof is complete.
\end{proof}

\begin{remark}
The above proof and result is similar to the Kru\v{z}kov interpolation lemma (see e.g. Lemma 2.4 in \cite{KaRi01}). For variants in the nonlocal case, we refer e.g. to \cite{CiJa11,CiJa14,DTEnJa17a,DTEnJa19}.
\end{remark}

\subsection{Results for the convection-diffusion problem}
\label{sec:ResultsForCDE}

All of the results that we have previously deduced also holds for \emph{entropy solutions} (cf. \cite{CiJa14}) of:

\begin{equation}\label{CDE}\tag{\textup{CDE}}
\begin{cases}
\dell_tu+\sum_{i=1}^N\dell_{x_i}f_i(u)-\mathfrak{L}[\varphi(u)]=g \qquad\qquad&\text{in $Q_T$,}\\
u(x,0)=u_0(x) \qquad\qquad&\text{on $\R^N$,}
\end{cases}
\end{equation}
where the convection and diffusion fluxes satisfy the standard assumption
\begin{align}
&f_i\in W_\textup{loc}^{1,\infty}(\R, \R)\text{ and }\varphi\in W_\textup{loc}^{1,\infty}(\R, \R)\text{ is nondecreasing}.
\tag{$\textup{A}_{f,\varphi}$}&
\label{Aconv}
\end{align}

To discretize the convection flux, we consider numerical fluxes $F:\R^2\to\R^N$ satisfying
\begin{enumerate}[{\rm (i)}]
\item $F=(F_1,\ldots, F_N)$ is Lipschitz continuous with Lipschitz constant $L_F$,
\item $F$ is consistent, i.e., $F_i(u,u)=f_i(u)$ for all $u\in\R$ and all $i\in\{1,\ldots,N\}$, and
\item $F(u_1,u_2)$ is nondecreasing w.t.r. $u_1$ and nonincreasing w.t.r. $u_2$,
\end{enumerate}
which gives a monotone approximation of $\sum_{i=1}^N\dell_{x_i}f_i(u)$ provided by
$$
\sum_{i=1}^ND_{h,i}^{-}F_i(\psi(x_\beta),\psi(x_\beta+he_i))\qquad\text{for all}\qquad \psi:\Grid\to\R,
$$
where $D_{h,i}^{-}\phi(x_\beta):=1/h(\phi(x_\beta)-\phi(x_\beta-he_i))$ and $\{e_i\}_{i=1}^N$ is the canonical basis of $\R^N$. The full scheme is then given by
\begin{equation}\label{defNumSchConv}
\begin{split}
U_\beta^j&=U_\beta^{j-1}-\Delta t_j\sum_{i=1}^ND_{h,i}^{-}F_i(U_{\beta}^{j-1},U_{\beta+e_i}^{j-1})\\
&\quad+\Delta t_j\Levy^{h}[\varphi(U_{\cdot}^j)]_\beta+\Delta t_jG_\beta^j\quad\forall(\beta,j)\in(\Z^N,\J),
\end{split}
\end{equation}
and it enjoys Propositions \ref{preResNumSch} and \ref{prop:AprioriTi}. We refer the reader to \cite{CiJa14} for further details. Analogously, Lemma \ref{lem:TailControlTi} also holds for \eqref{defNumSchConv}. It is then straightforward to check that Theorem \ref{tailControlGPME2} will contain the extra term
$$
C_{u_0,g,F}T\sum_{i=1}^N\|D_{h,i}^+\mathcal{X}_R\|_{L^p(\R^N)},
$$
where $D_{h,i}^+\phi(x):=1/h(\phi(x+he_i)-\phi(x))$ and $C_{u_0,g,F}>0$ is some constant depending on the norms of $u_0,g$ and the Lipschitz constant $L_F$ of $F$. That term can be estimated by Lemma \ref{convDAntiCutOff} and:

\begin{lemma}[{{\cite[Theorem 3]{Eva10}}}]\label{lem:estimateDiscConvection}
Assume $p\in[1,\infty]$, and let $\psi\in C^1(\R^N)$ with $D\psi\in L^p(\R^N)$. Then
\begin{equation*}
\begin{split}
\sup_{0<h<1}\sum_{i=1}^N\|D_{h,i}^+\psi\|_{L^p(\R^N)}\leq N\|D\psi\|_{L^p(\R^N)}.
\end{split}
\end{equation*}
\end{lemma}
Hence, the extra term converges to zero when $R\to\infty$ if $p>N$,
which poses no extra restriction in Corollary \ref{tailControlInLimitGPME2}. We then reproduce all the results of
Section \ref{sec:CompactnessAndConvergenceResults} for \eqref{CDE}
under the additional assumption \eqref{Aconv}. In particular:

\begin{theorem}[Compactness]
Assume \eqref{Aconv}--\eqref{muas}, and $r\in[1,\infty)$. Let $\{U_\beta^j\}_{h>0}$ be a sequence of a.e.-solutions of \eqref{defNumSchConv} and $\{\Uxt_h\}_{h>0}$ its corresponding sequence of interpolants. Then there exists a subsequence $\{\Uxt_{h_n}\}_{n\in\N}$ and a $u\in C([0,T];L^r(\R^N))$ such that
$$
\Uxt_{h_n}\to u \qquad\text{in $C([0,T];L^r(\R^N))$ as $n\to\infty$,}
$$
if any of the following holds:
\begin{enumerate}[{\rm (i)}]
\item When $\Levy^h=\Delta_h$. 
\item When $\Levy^h$ has weights satisfying \eqref{Dmuas'}.
\item When $\Levy^h$ has weights satisfying \eqref{DmuasFrac}.
\end{enumerate}
\end{theorem}

\begin{remark}
As mentioned in the introduction, we then improve the convergence results of \cite{EvKa00} and \cite{KaRi01} from the respective $C([0,T];L_\textup{loc}^1(\R^N))$ and $L_\textup{loc}^1(Q_T)$ to $C([0,T];L^1(\R^N))$. 
\end{remark}


\section*{Acknowledgements}

All the authors were supported by the Toppforsk (research excellence)
project Waves and Nonlinear Phenomena (WaNP), grant no. 250070 from the Research Council of Norway. 

F.~del Teso was supported by the Spanish Government through PGC2018-094522- B-I00, and RYC2020-029589-I funded by the MICIN/AEI.

J.~Endal has received funding from the European Union’s Horizon 2020 research and innovation programme under the Marie Sk\l odowska-Curie grant agreement no. 839749 ``Novel techniques for quantitative behavior of convection-diffusion equations (techFRONT)'', and from the Research Council of Norway under the MSCA-TOPP-UT grant agreement no. 312021.


\appendix

\section{Known compactness theorems}
\label{sec:AAKR}

We will use the following form of Arzel\`a-Ascoli's compactness theorem (see e.g. Chapter X.2.5 in \cite{Bou66} or Section 2 in \cite{Sim87}): 

\begin{theorem}[Arzel\`a-Ascoli]\label{thm:AA}
Let $(K,d_K)$ be a compact metric space and $(X,d_X)$ a complete metric space. Define 
$$
C(K; X):=\{\text{continuous functions $\psi:K\to X$}\}
$$ 
endowed with the uniform distance metric $\sup_{t\in K}d_X(\cdot,\cdot)$. Then $M\subset C(K;X)$ is relatively compact if and only if, for all $t\in K$, the following two conditions are fulfilled:
\begin{enumerate}[{\rm (i)}]
\item \textup{(Pointwise compactness)} $\{\psi(t)\in X:\psi\in M\}\subset X$ is relatively compact.
\item \textup{(Equicontinuity)} For all $\veps>0$, there exists a neighbourhood $V$ about $t$ such that, for all $\psi\in M$ and all $s\in V$,
$$
d_X(\psi(t),\psi(s))< \veps.
$$
\end{enumerate}
\end{theorem}

\begin{remark}\label{rem:ExtensionsAA}
\begin{enumerate}[{\rm (a)}]
\item A related version of the above theorem also exists. Define 
$$
\{\psi_n\}_{n\in\N}\subset F(K;X):=\{\text{functions $\psi:K\to X$}\}
$$ 
endowed with the uniform distance metric $\sup_{t\in K}d_X(\cdot,\cdot)$. Assume that $\{\psi_n\}_{n\in\N}$ satisfies condition (i) and condition (ii) replaced by: for all $\veps>0$ there exists $\eta>0$ such that
$$
d_K(t,s)<\eta\quad \textup{$\forall 0\leq t\leq s\leq T$}\quad\implies\quad \lim_{n\to\infty}d_X(\psi_n(t),\psi_n(s))< \veps\quad\textup{$\forall \psi_n\in M$}.
$$
Then there exists $\{\psi_{n_k}\}_{k\in\N}\in F(K;X)$ such that $\psi_{n_k}\to\psi$ in $F(K;X)$ as $k\to\infty$, and moreover, $\psi\in C(K;X)$. We refer the reader to \cite[Theorem C.11]{DrEyGaGuHe18} for a proof. 
\item Note that by part (a) it is possible to conclude that the limit is indeed in $C(K;X)$ even though the convergence is not in that space.
\end{enumerate}
\end{remark}

The standard Kolmogorov-Riesz's compactness theorem (see e.g. Theorem 5 in \cite{H-OHo10}) is given as follows:

\begin{theorem}[Kolmogorov-Riesz]\label{thm:KR}
Let $p\in[1,\infty)$. Then $M\subset L^p(\R^N)$ is relatively compact if and only if the following two conditions are fulfilled:
\begin{enumerate}[{\rm (i)}]
\item \textup{(Equicontinuity)} For all $\veps>0$ there is some $\eta>0$ such that, for all $\psi\in M$ and all $\xi\in\R^N$ with $|\xi|<\eta$,
$$
\|\psi(\cdot+\xi)-\psi\|_{L^p(\R^N)}< \veps.
$$ 
\item \textup{(Equitightness)} For all $\veps>0$ there is some $R>0$ such that, for all $\psi\in M$,
$$
\int_{|x|> R}|\psi(x)|^p\dd x<\veps^p.
$$
\end{enumerate}
\end{theorem}

\begin{remark}
Recall that in \cite{H-OHoMa17} it has been pointed out that the additional usual assumption of equiboundedness follows from properties (i)--(ii).
\end{remark}

\section{The Minkowski inequality for integrals}
\label{sec:Minkowski}

This is of course a well-known result (see e.g. Proposition 4.3 in \cite{DBe02}), but we include it for completeness.

\begin{theorem}\label{thm:Minkowski}
Let $(X,\mathcal{M},\mu)$ and $(Y,\mathcal{N},\nu)$ be $\sigma$-finite measure spaces and $p\in[1,\infty]$.
\begin{enumerate}[{\rm (a)}]
\item If $f\geq0$ is an $\mathcal{M}\otimes\mathcal{N}$-measurable function, then
\begin{enumerate}[{\rm (i)}]
\item the functions $x\mapsto\|f(x,\cdot)\|_{L^1(Y,\nu)}$ and $y\mapsto\|f(\cdot,y)\|_{L^p(X,\mu)}$ are respectively $\mu$- and $\nu$-measurable; and
\item 
\begin{equation*}
\begin{split}
&\bigg\|\int_{Y}f(\cdot,y)\dd \nu(y)\bigg\|_{L^p(X,\mu)}=\bigg(\int_{X}\bigg|\int_{Y}f(x,y)\dd \nu(y)\bigg|^p\dd\mu(x)\bigg)^{\frac{1}{p}}\\
&\leq\int_{Y}\bigg(\int_{X}|f(x,y)|^p\dd \mu(x)\bigg)^{\frac{1}{p}}\dd\nu(y)=\int_{Y}\|f(\cdot,y)\|_{L^p(x,\mu)}\dd \nu(y).
\end{split}
\end{equation*}
\end{enumerate}
\item If $f:X\times Y\to \mathbb{C}$ is an $\mathcal{M}\otimes\mathcal{N}$-measurable function and $\int_{Y}\|f(\cdot,y)\|_{L^p(X,\mu)}\dd \nu(y)<\infty$, then
\begin{enumerate}[{\rm (i)}]
\item the functions $x\mapsto\|f(x,\cdot)\|_{L^1(Y,\nu)}$ and $y\mapsto\|f(\cdot,y)\|_{L^p(X,\mu)}$ are respectively $\mu$- and $\nu$-measurable;
\item for $\mu$-a.e. $x$, $f(x,\cdot)\in L^1(Y,\dd \nu)$;
\item the $\mu$-a.e. defined function $x\mapsto \int_{Y}f(x,y)\dd \nu(y)$ is in $L^p(X,\mu)$; and
\item the inequality in (a)(ii) holds.
\end{enumerate}
\end{enumerate}
\end{theorem}

\begin{proof}
\noindent(a) The desired measurability follows by Section 252P in \cite{Fre01}, the rest is a consequence of the proof of Proposition 4.3 in \cite{DBe02}.

\medskip
\noindent(b) We apply part (a) to the function $|f|$.
\end{proof}


\section{Auxiliary results}
\label{sec:AuxRes}

\subsection{Uniform estimates for discrete operators}

In the local case $\Operator=\Delta$, we can always control the discretization by the operator itself. 

\begin{lemma}\label{lem:estimateDiscLaplace}
Assume $p\in[1,\infty]$. Let $\psi\in C^2(\R^N)$ with $\Delta\psi\in L^p(\R^N)$, and recall that
$$
\Delta_h\psi(x)=\sum_{i=1}^N\frac{\psi(x+he_i)+\psi(x-he_i)-2\psi(x)}{h^2}.
$$
Then
\begin{equation*}
\begin{split}
\|\Delta_h\psi\|_{L^p(\R^N)}\leq 2\|\Delta \psi\|_{L^p(\R^N)}.
\end{split}
\end{equation*}
\end{lemma}

\begin{remark}
Here we immediately see that control of $\Delta \psi$ gives control of $\Delta_h\psi$.
\end{remark}

\begin{proof}
An application of the Taylor theorem yields ($|he_i|=h<1$)
\begin{equation*}
\begin{split}
\Delta_h\psi(x)&=\sum_{i=1}^N\bigg(\int_{0}^1(1-t)D^2\psi(x+the_i)e_i\cdot e_i\dd t+\int_{-1}^0(1+t)D^2\psi(x+the_i)e_i\cdot e_i\dd t\bigg)\\
&=\int_{-1}^1(1-|t|)\Delta\psi(x+the_i)\dd t. 
\end{split}
\end{equation*}
Finally, by Jensen's inequality and Tonelli's lemma,
\begin{equation*}
\begin{split}
\|\Delta_h\psi\|_{L^p(\R^N)}&=\bigg(\int_{\R^N}|\Delta_h\psi(x)|^p\bigg)^{\frac{1}{p}}\leq\bigg(\int_{\R^N}\int_{-1}^1(1-|t|)^p|\Delta \psi(x+the_i)|^p\dd t \dd x\bigg)^{\frac{1}{p}}\\
&\leq 2\bigg(\int_{\R^N}|\Delta \psi(x)|^p \dd x\bigg)^{\frac{1}{p}}=2\|\Delta \psi\|_{L^p(\R^N)}.
\end{split}
\end{equation*}
The proof is complete.
\end{proof}

In the nonlocal case $\Operator=\Levy^\mu$, we cannot argue as above, but we can give estimates on its discretizations in an analogous way.

\begin{lemma}\label{lem:estimateDiscNonlocal}
Assume $p\in[1,\infty]$. Let  $\psi\in C^2(\R^N)$ with $\psi, D\psi, D^2\psi\in L^p(\R^N)$. Then:
\begin{enumerate}[{\rm (a)}]
\item Under $\sum_{\gamma\neq0}\omega_{\gamma,h}<\infty$ and for all $R>1$,
\begin{equation*}
\begin{split}
&\|\Levy^h[\psi]\|_{L^p(\R^N)}\\
&\leq\|D^2\psi\|_{L^p(\R^N)}\bigg\{\sum_{h\leq |z_\gamma|\leq R}|z_\gamma|^2\omega_{\gamma,h}\bigg\}+2\|\psi\|_{L^p(\R^N)}\bigg\{\sum_{|z_\gamma|>R}\omega_{\gamma,h}\bigg\}.
\end{split}
\end{equation*}
\item Under $\sum_{\gamma\neq0}\min\{|z_\gamma|^2,|z_\gamma|\}\omega_{\gamma,h}<\infty$,
\begin{equation*}
\begin{split}
&\|\Levy^h[\psi]\|_{L^p(\R^N)}\\
&\leq \|D^2\psi\|_{L^p(\R^N)}\bigg\{\sum_{h\leq |z_\gamma|\leq 1}|z_\gamma|^2\omega_{\gamma,h}\bigg\}+\|D\psi\|_{L^p(\R^N)}\bigg\{\sum_{|z_\gamma|>1}|z_\gamma|\omega_{\gamma,h}\bigg\}.
\end{split}
\end{equation*}
\end{enumerate}
\end{lemma}

\begin{remark}
We see that the estimation of $\|\Levy^h[\psi]\|_{L^p(\R^N)}$ only relies on information about $\|\psi\|_{L^p(\R^N)}$, $\|D\psi\|_{L^p(\R^N)}$, and $\|D^2\psi\|_{L^p(\R^N)}$. However, there is no immediate control of $\Levy^h[\psi]$ through $\Levy^\mu[\psi]$. Such control requires a deeper analysis, which we provide in Appendix \ref{sec:DiscreteAndContinuousEstimatesNonlocal} below. 
\end{remark}

\begin{proof}
\noindent(a) Recall that
$$
\Levy^h[\psi](x)=\sum_{\gamma\not=0} (\psi(x+z_\gamma)-\psi(x)) \omega_{\gamma,h}.
$$
As usual, since $\omega_{\gamma,h}=\omega_{-\gamma,h}$, we use
$$
\sum_{h\leq|z_\gamma|\leq R}z_\gamma\cdot D\psi(x)\mathbf{1}_{|z_\gamma|\leq R}\omega_{\gamma,h}=0
$$
to take advantage of the Taylor theorem:
$$
\sum_{h\leq|z_\gamma|\leq R} (\psi(x+z_\gamma)-\psi(x)) \omega_{\gamma,h}
=\sum_{h\leq|z_\gamma|\leq R}\bigg(\int_0^1(1-t)D^2\psi(x+tz_\gamma)z_\gamma\cdot z_\gamma\dd t\bigg)\omega_{\gamma,h}.
$$
We now apply a Minkowski type of inequality for integrals (cf. e.g. Theorem \ref{thm:Minkowski}) to get
\begin{equation*}
\begin{split}
\|\Levy^h[\psi]\|_{L^p(\R^N)}^p&\leq \int_{\R^N}\bigg|\sum_{h\leq |z_\gamma|\leq R}\Big(\psi(x+z_\gamma)-\psi(x)-z_\gamma\cdot D\psi(x)\mathbf{1}_{|z_\gamma|\leq R}\Big)\omega_{\gamma,h}\bigg|^p\dd x\\
&\quad+\int_{\R^N}\bigg|\sum_{|z_\gamma|>R}\big(\psi(x+z_\gamma)-\psi(x)\big)\omega_{\gamma,h}\bigg|^p\dd x\\
&\leq  \|D^2\psi\|_{L^p(\R^N)}\bigg\{\sum_{h\leq |z_\gamma|\leq R}|z_\gamma|^2\omega_{\gamma,h}\bigg\}+2\|\psi\|_{L^p(\R^N)}\bigg\{\sum_{|z_\gamma|>R}\omega_{\gamma,h}\bigg\}.
\end{split}
\end{equation*}

\medskip
\noindent(b) Here, we instead use the fundamental theorem of calculus when $|z_\gamma|>1$:
$$
\psi(x+z_\gamma)-\psi(x)=\int_0^1D\psi(x+tz_\gamma)\cdot tz_\gamma\dd t.
$$
Otherwise, the proof can then be completed in a similar manner as part (a).
\end{proof}

\subsection{Estimates for test function giving equitightness}
An immediate consequences of Lemmas \ref{lem:estimateDiscLaplace} and \ref{convDAntiCutOff} is: 

\begin{corollary}\label{wellPosednessLaplaceAntiCutOff}
Assume \eqref{Dmuas}, $\Levy^h=\Delta_h$, and let $\mathcal{X}_R$ be given by \eqref{antiCutOff1}--\eqref{antiCutOff2}.
If $p>N/2$, then
$$
\sup_{0<h<1}\|\Levy^h[\mathcal{X}_R]\|_{L^p(\R^N)}\to0\qquad\text{as $R\to\infty$}.
$$
\end{corollary}

The nonlocal case is more delicate.

\begin{corollary}\label{wellPosednessLevyAntiCutOff}
Let $\mathcal{X}_R$ be given by \eqref{antiCutOff1}--\eqref{antiCutOff2}.
Then:
\begin{enumerate}[{\rm (a)}]
\item Under \eqref{Dmuas}, $p>N/2$, and
\begin{equation}\label{eq:ExtraAssumptionDiscreteMeasure}
\sup_{0<h<1}\bigg\{R^{\frac{N}{p}-2}\sum_{1<|z_\gamma|\leq R}|z_\gamma|^2\omega_{\gamma,h}+R^{\frac{N}{p}}\sum_{|z_\gamma|>R}\omega_{\gamma,h}\bigg\}=o_R(1),
\end{equation}
we have
$$
\sup_{0<h<1}\|\Levy^h[\mathcal{X}_R]\|_{L^p(\R^N)}\to0\qquad\text{as $R\to\infty$}.
$$
\item Under \eqref{Dmuas'} and $p>N$, 
$$
\sup_{0<h<1}\|\Levy^h[\mathcal{X}_R]\|_{L^p(\R^N)}\to0\qquad\text{as $R\to\infty$}.
$$
\end{enumerate}
\end{corollary}

\begin{proof}
\noindent(a) Choose $\psi=1-\mathcal{X}_R$ in Lemma \ref{lem:estimateDiscNonlocal}(a) and recall Lemma \ref{convDAntiCutOff}  to obtain
\begin{equation*}
\begin{split}
\sup_{0<h<1}\|\Levy^h[\mathcal{X}_R]\|_{L^p(\R^N)}&=\sup_{0<h<1}\|\Levy^h[1-\mathcal{X}_R]\|_{L^p(\R^N)}\\
&\leq R^{\frac{N}{p}-2}\|D^2\mathcal{X}\|_{L^p(\R^N)}\sup_{0<h<1}\bigg\{\sum_{h\leq |z_\gamma|\leq 1}|z_\gamma|^2\omega_{\gamma,h}\bigg\}\\
&\quad+\Big(\|D^2\mathcal{X}\|_{L^p(\R^N)}+2\|1-\mathcal{X}\|_{L^p(\R^N)}\Big)\times\\
&\quad\times\sup_{0<h<1}\bigg\{R^{\frac{N}{p}-2}\sum_{1< |z_\gamma|\leq R}|z_\gamma|^2\omega_{\gamma,h}+R^{\frac{N}{p}}\sum_{|z_\gamma|>R}\omega_{\gamma,h}\bigg\}.
\end{split}
\end{equation*}
Both terms go to zero as $R\to\infty$ by $\eqref{Dmuas}$ and $p>N/2$, and \eqref{eq:ExtraAssumptionDiscreteMeasure}, respectively. 

\medskip
\noindent(b) Now, choose $\psi=1-\mathcal{X}_R$ in Lemma \ref{lem:estimateDiscNonlocal}(b) and recall Lemma \ref{convDAntiCutOff}  to obtain
\begin{equation*}
\begin{split}
\sup_{0<h<1}\|\Levy^h[\mathcal{X}_R]\|_{L^p(\R^N)}&=\sup_{0<h<1}\|\Levy^h[1-\mathcal{X}_R]\|_{L^p(\R^N)}\\&\leq \Big(R^{\frac{N}{p}-2}\|D^2\mathcal{X}\|_{L^p(\R^N)}+2R^{\frac{N}{p}-1}\|D\mathcal{X}\|_{L^p(\R^N)}\Big)\times\\
&\quad\times\sup_{0<h<1}\bigg\{\sum_{1< |z_\gamma|\leq R}|z_\gamma|^2\omega_{\gamma,h}+\sum_{|z_\gamma|>R}|z_\gamma|\omega_{\gamma,h}\bigg\}.
\end{split}
\end{equation*}
Since $p>N=\max\{N/2,N\}$ and \eqref{Dmuas'} holds, the right-hand side goes to zero as $R\to\infty$.
\end{proof}

\begin{example}\label{ex:measure1}
To get some intuition on the assumption \eqref{eq:ExtraAssumptionDiscreteMeasure}, let us think in terms of $\mu$, and check what happens for the choice
$$
\dd\mu_{\beta,\alpha}(z):=\frac{\dd z}{|z|^{N+\beta}}\mathbf{1}_{0<|z|\leq1}+\frac{\dd z}{|z|^{N+\alpha}}\mathbf{1}_{|z|>1}
$$
with $\beta,\alpha\in(0,2)$. As this measure satisfies \eqref{muas}, it is a L\'evy measure, and $\int_{0<|z|\leq 1}|z|^{2}\dd\mu_{\beta,\alpha}(z)$ is simply finite. Thus, the tail behaviour is more interesting to us. One can check that
\begin{equation*}
\begin{split}
R^{\frac{N}{p}-2}\int_{1<|z|\leq R}|z|^2\dd\mu_{\beta,\alpha}(z)&=\frac{1}{2-\alpha}\big(R^{\frac{N}{p}-\alpha}-R^{\frac{N}{p}-2}\big)=\mathcal{O}(R^{\frac{N}{p}-\alpha})\\
\end{split}
\end{equation*}
 and furthermore,
$$
R^{\frac{N}{p}}\int_{|z|>R}\dd \mu_{\beta,\alpha}(z)=\mathcal{O}(R^{\frac{N}{p}-\alpha}).
$$
To ensure that $\Levy^{\mu_{\beta,\alpha}}[\mathcal{X}_R]\to0$ in $L^p(\R^N)$ as $R\to\infty$, we let $p>N/\alpha>N/2$ when $N\geq 2$, $p>1/\alpha$ when $\alpha\in (0,1)$ and $N=1$, and $p>1$ when $\alpha\in[1,2)$ and $N=1$. 

Now, for any $\mu$ satisfying \eqref{muasFrac}, we still have
$$
R^{\frac{N}{p}-2}\int_{1<|z|\leq R}|z|^2\dd\mu(z)=R^{\frac{N}{p}-\alpha}R^{\alpha-2}\int_{1<|z|\leq R}|z|^2\dd\mu(z)\leq CR^{\frac{N}{p}-\alpha}
$$
and
$$
R^{\frac{N}{p}}\int_{|z|>R}\dd \mu(z)=R^{\frac{N}{p}-\alpha}R^{\alpha}\int_{1<|z|\leq R}\dd\mu(z)\leq CR^{\frac{N}{p}-\alpha}.
$$
Hence, the mentioned $L^p(\R^N)$-convergence still holds under this assumption as well.
\end{example}

\subsection{Consistency and uniform estimates in the nonlocal case}
\label{sec:DiscreteAndContinuousEstimatesNonlocal}

In the local case, one can always relate the discrete operator with the limit operator $\Delta$, as in Lemma \ref{lem:estimateDiscLaplace}. This is not that clear in the nonlocal case. However, we will show that we can check \eqref{Dmuas}, \eqref{Dmuas'}, and \eqref{DmuasFrac} through information about $\Levy^\mu$.

Proposition 3.2 in \cite{DTEnJa18b} provides the following result: Assumption \eqref{Dmuas} holds if \eqref{muas} holds and, for all $\psi\in C_\textup{c}^\infty(\R^N)$,
\begin{equation}\label{eq:consinf}
\Levy^h[\psi]\to \Operator[\psi] \qquad\textup{pointwise as $h\to0^+$.}
\end{equation}
The proof is based on choosing the particular test function
$$
\psi(x):=\begin{cases}
|x|^2-1 \qquad\qquad&\text{if $|x|\leq 1$}\\
\phi(x) \qquad\qquad&\text{if $|x|> 1$,}
\end{cases}
$$
where $\phi$ is some nonnegative function such that $\psi\in C_\textup{c}^\infty(\R^N)$. We will now provide a similar result in the case \eqref{Dmuas'}:

\begin{lemma}
Assumption \eqref{Dmuas'} holds if \eqref{muas'} holds and
for
$$
\psi(x):=\begin{cases}
|x|^2-1 \qquad\qquad&\text{if $|x|\leq 1$}\\
2(|x|-1) \qquad\qquad&\text{if $|x|> 1$,}
\end{cases}
$$
we have \eqref{eq:consinf}.
\end{lemma}

\begin{proof}
We follow the proof of Proposition 3.2 in \cite{DTEnJa18b}. In our particular case,
$$
D\psi(x)=\begin{cases}
2x \qquad\qquad&\text{if $|x|\leq 1$}\\
2\frac{x}{|x|} \qquad\qquad&\text{if $|x|> 1$,}
\end{cases}
\qquad\text{and}\qquad
D^2\psi(x)=\begin{cases}
2 \qquad\qquad&\text{if $|x|\leq 1$}\\
0 \qquad\qquad&\text{if $|x|> 1$.}
\end{cases}
$$
We get 
\begin{equation*}
\begin{split}
\Levy^{\nu_h}[\psi](0)&=\int_{|z|\leq 1}\int_0^1(1-t)z^TD^2\psi(tz)z\dd t\dd\nu_h(z)+\int_{|z|>1}\int_0^1D\psi(tz)\cdot tz\dd t\dd\nu_h(z)\\
&=\int_{|z|\leq 1}\int_0^1(1-t)z^TD^2\psi(tz)z\dd t\dd\nu_h(z)+\int_{1<|z|\leq 1/t}\int_0^1D\psi(tz)\cdot tz\dd t\dd\nu_h(z)\\
&\quad+\int_{|z|>1/t}\int_0^1D\psi(tz)\cdot tz\dd t\dd\nu_h(z)\\
&=\int_{|z|\leq 1}|z|^2\dd\nu_h(z)+\int_{1<|z|\leq 1/t}\int_0^12t^2|z|^2\dd t\dd\nu_h(z)+\int_{|z|>1/t}\int_0^12t|z|\dd t\dd\nu_h(z)\\
&\geq \int_{|z|\leq 1}|z|^2\dd\nu_h(z)+\int_{1<|z|\leq 1/t}\int_0^12t^2\dd t|z|\dd\nu_h(z)+\int_{|z|>1/t}\int_0^12t\dd t|z|\dd\nu_h(z)\\
&\geq \int_{|z|>0}\min\{|z|^2,|z|\}\dd \nu_h(z).
\end{split}
\end{equation*}
Hence, 
\begin{equation*}
\begin{split}
\sup_{0<h<1}\int_{|z|>0}\min\{|z|^2,|z|\}\dd \nu_h(z)&\leq \sup_{0<h<1}\Levy^{\nu_h}[\psi](0)\\
&\leq \|\Levy^\mu[\psi]\|_{L^\infty(\R^N)}+\sup_{0<h<1}|(\Levy^{\nu_h}-\Levy^\mu)[\psi](0)|,
\end{split}
\end{equation*}
which is finite due to the pointwise convergence and the assumption \eqref{muas'}. 
\end{proof}

In the case \eqref{DmuasFrac}, we get:
\begin{lemma}
Assumption \eqref{DmuasFrac} holds if $\mu$ satisfies \eqref{muasFrac} and for
$$
\psi(x):=\begin{cases}
R^{\alpha-2}|x|^2-R^\alpha \qquad\qquad&\text{if $|x|\leq R$}\\
\phi(x) \qquad\qquad&\text{if $|x|> R$,}
\end{cases}
$$
where $\phi$ is some nonnegative function such that $\psi\in C_\textup{c}^\infty(\R^N)$ and $\|\psi\|_{L^\infty(\R^N)}=|\psi(0)|=R^\alpha$, we have \eqref{eq:consinf} uniformly in $R$.
\end{lemma}

\begin{proof}
Again, we follow the proof of Proposition 3.2 in \cite{DTEnJa18b}. We obtain
\begin{equation*}
\begin{split}
\Levy^{\nu_h}[\psi](0)&\geq R^{\alpha-2}\int_{|z|\leq 1}|z|^2\dd\nu_h(z)+R^{\alpha-2}\int_{1<|z|\leq R}|z|^2\dd\nu_h(z)-\psi(0)\int_{|z|>R}\dd \nu_h(z)\\
&\geq R^{\alpha-2}\int_{1<|z|\leq R}|z|^2\dd\nu_h(z)+R^\alpha\int_{|z|>R}\dd \nu_h(z).
\end{split}
\end{equation*}
Hence, since $R>1$ and $\alpha\in(0,2)$, we have $R^{\alpha-2}\leq 1$ and
\begin{equation*}
\begin{split}
&\sup_{0<h<1}\bigg\{R^{\alpha-2}\int_{1<|z|\leq R}|z|^2\dd\nu_h(z)+R^\alpha\int_{|z|>R}\dd \nu_h(z)\bigg\}\\
&\leq \sup_{0<h<1}\Levy^{\nu_h}[\psi](0)\leq |\Levy^\mu[\psi](0)|+\sup_{0<h<1}|(\Levy^{\nu_h}-\Levy^\mu)[\psi](0)|\\
&\leq \bigg(R^{\alpha-2}\int_{|z|\leq 1}|z|^2\dd \mu(z)+R^{\alpha-2}\int_{1<|z|\leq R}|z|^2\dd\mu(z)+\int_{|z|>R}|\psi(z)-\psi(0)|\dd \mu(z)\bigg)\\
&\quad+\sup_{0<h<1}|(\Levy^{\nu_h}-\Levy^\mu)[\psi](0)|\\
&\leq \bigg(\int_{|z|\leq 1}|z|^2\dd \mu(z)+R^{\alpha-2}\int_{1<|z|\leq R}|z|^2\dd\mu(z)+2R^\alpha\int_{|z|>R}\dd \mu(z)\bigg)\\
&\quad+\sup_{0<h<1}|(\Levy^{\nu_h}-\Levy^\mu)[\psi](0)|,
\end{split}
\end{equation*}
which is finite due to the pointwise convergence and the assumption \eqref{muasFrac}. 
\end{proof}









\end{document}